\documentclass[12pt]{article}
\usepackage{amssymb,amsmath,amsthm, amsfonts}
\usepackage{graphicx}
\usepackage{epsfig}
\usepackage{tikz}

\def\disp{\displaystyle}

\def\dref#1{(\ref{#1})}

\theoremstyle{plain}
\newtheorem{theorem}{Theorem}[section]
\newtheorem{lemma}{Lemma}[section]

\theoremstyle{definition}

\newtheorem{remark}{Remark}[section]

\numberwithin{equation}{section}

\begin{document}

\title{\bf A new (and optimal) result   for boundedness of solution of a quasilinear   chemotaxis--haptotaxis model (with  logistic source)}

\author{Ling Liu $^{1}$, 
Jiashan Zheng$^{a}$\thanks{Corresponding author.   E-mail address:
 zhengjiashan2008@163.com (J. Zheng)}, Yu Li$^{a}$, Weifang Yan$^{a}$
\\
 $^{1}$
    Department of Basic Science,\\
    Jilin Jianzhu University, Changchun 130118, P.R.China \\
 $^{a}$
 School of Mathematics and Statistics Science,\\
     Ludong University, Yantai 264025,  P.R.China \\
}
\date{}

\maketitle \vspace{0.3cm}
\noindent
\begin{abstract}
This article deals with an initial-boundary value problem for the coupled chemotaxis-haptotaxis
system with nonlinear diffusion
$$
 \left\{\begin{array}{ll}
  u_t=\nabla\cdot( D(u)\nabla u)-\chi\nabla\cdot(u\nabla v)-
  \xi\nabla\cdot(u\nabla w)+\mu u(1- u-w),
x\in \Omega, t>0,\\
 \disp{\tau v_t=\Delta v- v +u},\quad
x\in \Omega, t>0,\\
\disp{w_t=- vw},\quad
x\in \Omega, t>0,\\
 \end{array}\right.\eqno(0.1)
$$
under homogeneous Neumann boundary conditions in a smooth bounded domain  $\Omega\subset\mathbb{R}^N(N\geq1)$, where $\tau\in\{0,1\}$ and $\chi$, $\xi$ and $\mu$ are given nonnegative parameters. 
The diffusivity
$ D(u)$ is assumed to satisfy
$$  D(u) \geq C_{D}(u+1)^{m-1}~ \mbox{for all}~ u\geq0~~\mbox{and}~~C_{D}>0.$$
In the present work it is shown that
if
$$
m\geq 2-\frac{2}{N}\lambda~~\mbox{with}~~0<\mu<\kappa_0,
\begin{array}{ll}\\
 \end{array}
$$
$$
m> 2-\frac{2}{N}\lambda~~\mbox{with}~~\mu\geq\kappa_0,
\begin{array}{ll}\\
 \end{array}
$$
or
$$m> 2-\frac{2}{N}~~\mbox{and}~~~\mu=0$$
or
$$m= 2-\frac{2}{N}~~\mbox{and}~~~C_{D} >\frac{C_{GN}(1+\|u_0\|_{L^1(\Omega)})^3}{4}(2-\frac{2}{N})^2\kappa_0,$$
then for all reasonably regular initial data, a corresponding initial-boundary value
problem for $(0.1)$ possesses a unique global classical solution that is uniformly bounded in $\Omega\times(0, \infty)$, where
$$
 \lambda=\frac{\kappa_0}{(\kappa_0-\mu)_+}
$$
and
$$
 \kappa_0=\left\{\begin{array}{ll}
\max_{s\geq1}\lambda_0^{\frac{1}{{{s}}+1}}(\chi+\xi\|w_0\|_{L^\infty(\Omega)})~~ \tau=1,\\
\chi~~~~~~\mbox{if}~~ \tau=0.\\
 \end{array}\right.
$$
Here  $C_{GN}$ and $\lambda_0$ are the  constants which are corresponding to the Gagliardo--Nirenberg  inequality (see Lemma \ref{lemma41}) and the maximal Sobolev
regularity (see Lemma \ref{lemma45xy1222232}), respectively.
Relying on a {\bf new} $L^p$-estimate techniques to raise the a priori
estimate of a solution from $L^1 (\Omega) \rightarrow L^{\lambda-\varepsilon} (\Omega) \rightarrow L^{\lambda} (\Omega)\rightarrow
L^{\lambda+\varepsilon} (\Omega)\rightarrow L^{p} (\Omega) (\mbox{for all}~p>1)$,
these results  thereby significantly extending   results of previous results of several authors (see Remarks 1.1 and 1.2) and some optimal results are obtained.

\end{abstract}

\vspace{0.3cm}


\newpage
\section{Introduction}

 In order to  describe the cancer cell invasion into surrounding
healthy tissue,
  in 2005, Chaplain and Lolas (\cite{Chaplain1})
 proposed a pioneering mathematical model which is called chemotaxis--haptotaxis model
 \begin{equation}
 \left\{\begin{array}{ll}
  u_t=\nabla\cdot(D\nabla u)-\chi\nabla\cdot(u\nabla v)-\xi\nabla\cdot
  (u\nabla w)+\mu u(1-u-w),\quad
x\in \Omega, t>0,\\
 \disp{\tau v_t=\Delta v +u- v},\quad
x\in \Omega, t>0,\\
\disp{w_t=- vw },\quad
x\in \Omega, t>0,\\
 \end{array}\right.\label{7101.sdddd2x19ssddx3189}
\end{equation}
where $D$, $\chi,\xi$ and $\mu$ are the cancer cell random
motility, the  chemotactic coefficients, the  haptotactic coefficients and the proliferation rate of the cells,
respectively.
Here $\tau\in\{0,1\},$
 $\Omega \subset \mathbb{R}^N (N\geq1)$ is the physical domain which we assume to be bounded with
smooth boundary, and  the unknown quantities $u$, $v$ and $w$ represent the density of
cancer cells, the concentration of  matrix degrading enzymes (MDE) and the density of extracellular matrix (ECM),
respectively.

As a subsystem, \dref{7101.sdddd2x19ssddx3189} contains the celebrated Keller--Segel (\cite{Keller2710}) chemotaxis system  (with logistic source, $\mu\neq0$)
 \begin{equation}\label{x1.7dffddffff31426677gghhhjgg}
 \left\{\begin{array}{ll}
  u_t=\nabla\cdot(D\nabla u)-\chi\nabla\cdot(u\nabla v)+\mu u(1-u),\quad
x\in \Omega, t>0,\\
 \disp{\tau v_t=\Delta v +u- v},\quad
x\in \Omega, t>0\\
 \end{array}\right.
\end{equation}
by setting $w\equiv 0$.
 Over the last four decades, there is a wide variety of patterns associated Keller--Segel system \dref{x1.7dffddffff31426677gghhhjgg}  have been studied extensively, and the main interest lies in whether the solution
is global or blow up (see e.g., Cie\'{s}lak \cite{Cie791}, Cie\'{s}lak and Winkler \cite{Cie72}, Ishida et al. \cite{Ishida}, Painter and Hillen \cite{Painter79}, Winkler \cite{Winkler37103,Winkler79312,Winkler79312,Winkler715}, Li and  Xiang \cite{LiLittffggsssdddssddxxss}, Tello and  Winkler \cite{Tello710}, Wang et al. \cite{Wang76}, Zheng et al. \cite{Zhengssdddssddddkkllssssssssdefr23}).
In fact, if $\mu=0,$
the two behaviors (boundedness and blow-up)
of solutions strongly depend on the space dimension and the total mass of cells (\cite{Biler444710,Horstmann79,Horstmann791,Winkler792}).
When $\tau = 0$, 
Tello and  Winkler (\cite{Tello710})  mainly proved that
the global boundedness for model \dref{7101.sdddd2x19ssddx3189} exists
 under the condition $\mu > \frac{(N-2)^+}{N}\chi$, moreover,  they gave the  weak solutions
 for arbitrary small $\mu > 0$.  Kang and Stevens \cite{Kangaffrk}
  (see also \cite{Xiangssffrk,Huaffrk})  improve
the results of Tello and  Winkler (\cite{Tello710})
 to the case $ \mu \geq {\bf\frac{(N-2)_+}{N}\chi}$.
While if $\tau=1$  and 
$\mu>\frac{(N-2)_+}{N}\chi C_{\frac{N}{2}+1}^{\frac{1}{\frac{N}{2}+1}}$ (where $C_{\frac{N}{2}+1}$ is a positive constant), Zheng (\cite{Zhengssdddssddddkkllssssssssdefr23}) proved that
for any sufficiently smooth initial data, the corresponding initial-boundary
value problem for \dref{x1.7dffddffff31426677gghhhjgg} possesses a globally defined bounded solution,
 which give the lower bound estimation for the logistic source, so that, improves the result of \cite{Winkler37103}.
Furthermore,  some recent studies have shown that
the blow-up of solutions can be inhibited by the nonlinear diffusion (see   {Ishida} et al. \cite{Ishida},  Winkler et al. \cite{Bellomo1216,Tao794,Zheng33312186,Zheng0,Winkler79,Winkler72}) and nonlinear logistic term (see \cite{Zheng0,Zhengsssddswwerrrseedssddxxss}).

There have been large literature on the global existence and the large time behavior of
solutions to the system \dref{7101.sdddd2x19ssddx3189}. We refer to \cite{Cao,Marciniak,Tao2,Taox26,Taox26216,Taoddfvbb41215,Zhengddfggghjjkk1} and the references therein.
%
%
In fact, when $\tau = 0$, MDEs diffuses much faster than cells (see \cite{Jger317,Taox26216}), 
Tao and Wang
\cite{Tao2} proved that model \dref{7101.sdddd2x19ssddx3189} possesses a unique global bounded classical solution for
any $\mu > 0$ in two space dimensions, and for large $\mu > 0$ in three space dimensions.
In \cite{Taox26216},Tao and Winkler improved the condition on $\mu$ (${\bf \mu > \frac{(N-2)^+}{N}\chi}$), so that it coincides with the best one known for
the parabolic-elliptic Keller-Segel system \dref{x1.7dffddffff31426677gghhhjgg} (see Tello and  Winkler\cite{Tello710}), moreover, in additional explicit
smallness on $w_0$, they gave the exponential decay of $w$ in the large time limit. However,   this problem is left open for
the critical case ${\bf \mu =\frac{(N-2)^+}{N}\chi}$.
While, if $\tau=1$,  
refined
approaches involving a more subtle analysis of \dref{7101.sdddd2x19ssddx3189}, 
Tao (\cite{Taox3201}) and  Cao  (\cite{Cao}) obtained
the boundedness  of global solution for the  2D and 3D space respectively, especially, for the 3D
space, similar to the chemotaxis-only system (\cite{Zhengssdddssddddkkllssssssssdefr23,Winkler37103}), the global solution is obtained only for large $\mu$,
and it remains open for small $\mu$.

The diffusion of cancer cell may depend nonlinearly on their densities (\cite{Hillen79,Szymaska,Tao72}), and so we are led to consider the cell motility $D$ as
a nonlinear function of the cancer cell density, $D\equiv   D(u)= C_ D(u + 1)^{m-1}, m \in\mathbb{R},  C_{D} > 0$.
 Introducing
this into the model \dref{7101.sdddd2x19ssddx3189} leads to the following chemotaxis-haptotaxis system with {\bf nonlinear diffusion}
\begin{equation}
 \left\{\begin{array}{ll}
  u_t=\nabla\cdot( D(u)\nabla u)-\chi\nabla\cdot(u\nabla v)-\xi\nabla\cdot
  (u\nabla w)+\mu u(1-u-w),\quad
x\in \Omega, t>0,\\
 \disp{\tau v_t=\Delta v +u- v},\quad
x\in \Omega, t>0,\\
\disp{w_t=- vw },\quad
x\in \Omega, t>0,\\
 \disp{D(u)\frac{\partial u}{\partial \nu}-\chi\frac{\partial v}{\partial \nu}-\xi\frac{\partial w}{\partial \nu}=\frac{\partial v}{\partial \nu}=0},\quad
x\in \partial\Omega, t>0,\\
\disp{u(x,0)=u_0(x)},\tau v(x,0)=\tau v_0(x),w(x,0)=w_0(x),\quad
x\in \Omega,\\
 \end{array}\right.\label{1.1}
\end{equation}
where  $u,v$ and $w$ are denoted as before,
 $\mu\geq0,\tau\in\{0,1\}$,
the diffusion function $D(u)$
 fulfills
\begin{equation}\label{9161}
D\in  C^{2}([0,\infty))
\end{equation}
and there
exist  constants $m\in \mathbb{R}$ and  $C_{D}$
such that
\begin{equation}\label{9162}
 D(u) \geq C_{D}(u+1)^{m-1}~ \mbox{for all}~ u\geq0.
\end{equation}
This parabolic-parabolic-ODE
system ($\tau=1$ in \dref{1.1}) and its parabolic-elliptic-ODE simplifications ($\tau=0$ in \dref{1.1}) have been objects of extensive studies in recent decades.
In fact, in $N=2$, Zheng et al. (\cite{Zhenghhjjghjjkk1}) studied the global boundedness for model \dref{1.1}
with $D$ satisfies \dref{9161}--\dref{9162} and
$m>1$, moreover, in additional explicit
smallness on $w_0$, they gave the exponential decay of $w$ in the large time limit.
Moreover, if $D$ satisfies \dref{9161}--\dref{9162} with
$m>\max\{1,\bar{m}\}$
and
\begin{equation}\label{dcfvgg7101.2x19x318jkl}
\bar{m}:=\left\{\begin{array}{ll}
\frac{2N^2+4N-4}{N(N+4)}~~\mbox{if}~~
N \leq 8,\\
 \frac{2N^2+3N+2-\sqrt{8N(N+1)}}{N(N+20)}~~\mbox{if}~~
N \geq 9,\\
 \end{array}\right.
\end{equation}
Tao and Winkler (\cite{Tao72}) proved that model \dref{1.1} possesses at least one nonnegative
global classical solution, however, their boundedness is left as an open problem.
Using the boundedness of $\int_{\Omega}|\nabla v|^{l}(1\leq l<\frac{N}{N-1})$,
 Wang (\cite{Wangscd331629}) and Li, Lankeit (\cite{Li445666})  proved that
 the
global solvability and boundedness  of classical solution  (or weak solution) for
any 
 $D$ satisfies \dref{9161}--\dref{9162} and
$m>2-\frac{2}{N}$.
Recently, Zheng (\cite{Zhddengssdeeezseeddd0}) and Jin (\cite{Jinhhyyyyy}) extended these results to the case $m > \frac{2N}{N+2}$ and $m>0$ (as well as  large $\mu$), respectively.  But the cases
$m\leq0$ remain unknown.
Other variants of the model that are commonly treated including
the (nonlinear) logistic
types and  the re-establishment of ECM components, please refer to
\cite{Stinnerddff12,PangPang1,Tao79477,Zhenssdssdddfffgghjjkk1}, etc, and
references therein. Thus it is meaningful to analyze the following question:

$(Q)$:  Which size of $m$, $\chi,\xi$ and $\mu$ are sufficient to ensure boundedness of solutions to \dref{1.1}?

It is our goal in this work to give answers to $(Q)$.
%
%
 To the best of our knowledge, this is the first result which gives  a {\bf explicit condition}  between $m,\chi,\xi$ and $\mu$ that yields to the boundedness
   of the solution.

Motivated by the above works,
  the aim of the present paper is to
   study the quasilinear parabolic--elliptic--ODE ($\tau=0$ in \dref{1.1}) and parabolic--parabolic--ODE ($\tau=1$ in \dref{1.1})  chemotaxis--haptotaxis model  \dref{1.1} under the conditions \dref{9161}--\dref{9162}. Our main result is the following:

   \begin{theorem}\label{theorem3}Let  $\Omega\subset \mathbb{R}^N(N\geq1)$ be
 a bounded domain with smooth boundary $\partial\Omega$. Assume
that $D$   satisfy \dref{9161}--\dref{9162} and the initial data $(u_0,w_0)$ fulfills
 \begin{equation}\label{aassx1.731426677gg}
\left\{
\begin{array}{ll}
\displaystyle{u_0\in C(\bar{\Omega})~~\mbox{with}~~u_0\geq0~~\mbox{in}~~\Omega~~\mbox{and}~~u_0\not\equiv0},\\
\displaystyle{w_0\in C^{2+\vartheta}(\bar{\Omega})~~\mbox{with}~~w_0\geq0~~\mbox{in}~~\bar{\Omega}~~\mbox{and}~~\frac{\partial w_0}{\partial\nu}=0~~\mbox{on}~~\partial\Omega} ~~~~~~~~~~~~~~~~~~~~~~~~~~~~~~~~~~~~~~~~~~~~~~~~~~~~~~~~~~~~~~~~~~~\\
\end{array}
\right.
\end{equation}
with some $\vartheta\in(0,1).$

If one of the following cases holds:

  (i)~~$m\geq2-\frac{2}{N}\frac{\chi}{(\chi-\mu)_{+}}$ with $0<\mu<\chi$;

   (ii)~~$m>2-\frac{2}{N}\frac{\chi}{(\chi-\mu)_{+}}$  with $\mu\geq\chi$;

  (iii)~~$m>2-\frac{2}{N}$ with $\mu=0$;

(iv)~~$m=2-\frac{2}{N}$ and $C_{D} >\frac{C_{GN}(1+\|u_0\|_{L^1(\Omega)})^3}{4}(2-\frac{2}{N})^2\chi$;

%
%
%
%
then there exists a triple $(u,v,w)\in (C^0(\bar{\Omega}\times[0,\infty))\cap C^{2,1}
(\bar{\Omega}\times(0,\infty)))^3$ which solves \dref{1.1} in the classical sense. Here $C_{GN}$ is a  positive   constant which is corresponding to the Gagliardo--Nirenberg  inequality (see Lemma \ref{lemma41}).
Moreover, both $u$, $v$  and $w$ are bounded in $\Omega\times(0,\infty)$.

\end{theorem}

Before we prove theorem \ref{theorem3}, there exist  a few remarks in order.
\begin{remark}
(i) 
 Theorem \ref{theorem3}  extends the results of  Theorem 1.1  of Tao and  Winkler (\cite{Taox26216})
for the critical case $\mu=\frac{(N-2)_{+}}{N}\chi$ and $D(u)=1.$

(ii) If $\mu=0,$ in comparison to the result
for the corresponding haptotaxis-free system (\cite{Winkler72}, $w\equiv0$), it is easy to see that the restriction on $m$ here is
{\bf optimal}.


(iii) Observing that if $\mu\geq\frac{(N-2)_{+}}{N}\chi$ and $w\equiv0$, then
$2-\frac{2}{N}\frac{\chi}{(\chi-\mu)_{+}}<1$, therefore,  Theorem \ref{theorem3} also
extends the results of  Theorem 1.1  of Tello and  Winkler (\cite{Tello710}). 

(iv) Obviously, if $\mu>\chi$, then
$2-\frac{2}{N}\frac{\chi}{(\chi-\mu)_{+}}<1$, so that,  Theorem \ref{theorem3}
extends the results of  Theorem 1.1  of Tao and  Winkler (\cite{Taox26}).

(v) Obviously, if $w\equiv0$ and $\mu>0,$ then
$2-\frac{2}{N}\frac{\chi}{(\chi-\mu)_{+}}<2-\frac{2}{N}$, so that,  Theorem \ref{theorem3} also
partly extends the results of  Theorem 1.1  of Wang et al. (\cite{Wang76}).

(vi) Obviously, if $w\equiv0$,  $\mu\geq\frac{(N-2)_{+}}{N}\chi$ and $D(u)\equiv1$, then
$2-\frac{2}{N}\frac{\chi}{(\chi-\mu)_{+}}\leq1$, so that,  Theorem \ref{theorem3} is consistent with the results of Theorem 3 of  \cite{Kangaffrk}.

\end{remark}

\begin{theorem}\label{dfffftheorem3} Let  $\Omega\subset \mathbb{R}^N(N\geq1)$ be
 a bounded domain with smooth boundary $\partial\Omega$.
Assume
that $D$   satisfy \dref{9161}--\dref{9162} and the initial data $(u_0,v_0,w_0)$ fulfills
 \begin{equation}\label{x1.731426677gg}
\left\{
\begin{array}{ll}
\displaystyle{u_0\in W^{1,\infty}(\Omega)~~\mbox{with}~~u_0\geq0~~\mbox{in}~~\Omega~~\mbox{and}~~u_0\not\equiv0},\\
\displaystyle{v_0\in W^{1,\infty}(\Omega)~~\mbox{with}~~v_0\geq0~~\mbox{in}~~\Omega~~\mbox{and}~~\frac{\partial v_0}{\partial\nu}=0~~\mbox{on}~~\partial\Omega},\\
\displaystyle{w_0\in C^{2+\vartheta}(\bar{\Omega})~~\mbox{with}~~w_0\geq0~~\mbox{in}~~\bar{\Omega}~~\mbox{and}~~\frac{\partial w_0}{\partial\nu}=0~~\mbox{on}~~\partial\Omega} ~~~~~~~~~~~~~~~~~~~~~~~~~~~~~~~~~~~~~~~~~~~~~~~~~~~~~~~~~~~~~~~~~~~\\
\end{array}
\right.
\end{equation}
with some $\vartheta\in(0,1).$
  If one of the following cases holds:

(i)~~$m\geq2-\frac{2}{N}\frac{\max_{s\geq1}\lambda_0^{\frac{1}{{{s}}+1}}(\chi+\xi\|w_0\|_{L^\infty(\Omega)})}{[\max_{s\geq1}\lambda_0^{\frac{1}{{{s}}+1}}(\chi+\xi\|w_0\|_{L^\infty(\Omega)})-\mu]_{+}}$  with $0<\mu<\max_{s\geq1}\lambda_0^{\frac{1}{{{s}}+1}}(\chi+\xi\|w_0\|_{L^\infty(\Omega)})$;

(ii)~~$m>2-\frac{2}{N}\frac{\max_{s\geq1}\lambda_0^{\frac{1}{{{s}}+1}}(\chi+\xi\|w_0\|_{L^\infty(\Omega)})}{[\max_{s\geq1}\lambda_0^{\frac{1}{{{s}}+1}}(\chi+\xi\|w_0\|_{L^\infty(\Omega)})-\mu]_{+}}$ with $\mu\geq\max_{s\geq1}\lambda_0^{\frac{1}{{{s}}+1}}(\chi+\xi\|w_0\|_{L^\infty(\Omega)})$;

(iii)~~$m>2-\frac{2}{N}$ with $\mu=0$;

(iv)~~$m=2-\frac{2}{N}$ and $C_{D} >\frac{C_{GN}(1+\|u_0\|_{L^1(\Omega)})^3}{4}(2-\frac{2}{N})^2\max_{s\geq1}\lambda_0^{\frac{1}{{{s}}+1}}(\chi+\xi\|w_0\|_{L^\infty(\Omega)})$;
%
%
%
%
%
%

then there exists a pair $(u,v,w)\in (C^0(\bar{\Omega}\times[0,\infty))\cap C^{2,1}(\bar{\Omega}\times(0,\infty))^2$ which solves \dref{1.1} in the classical sense, where $C_{GN}$ and $\lambda_0:=\lambda_0(\gamma)$ are the  constants which are corresponding to the Gagliardo--Nirenberg  inequality (see Lemma \ref{lemma41}) and the maximal Sobolev
regularity (see Lemma \ref{lemma45xy1222232}), respectively.
Moreover, both $u$ and $v$ are bounded in $\Omega\times(0,\infty)$.
\end{theorem}

\begin{remark}
(i) 
Obviously, if $\mu>\frac{(N-2)_{+}}{N}\chi \max_{s\geq1}\lambda_0^{\frac{1}{{{s}}+1}},$
$2-\frac{2}{N}\frac{\max_{s\geq1}\lambda_0^{\frac{1}{{{s}}+1}}(\chi+
\xi\|w_0\|_{L^\infty(\Omega)})}{[\max_{s\geq1}\lambda_0^{\frac{1}{{{s}}+1}}(\chi+\xi\|w_0\|_{L^\infty(\Omega)})-\mu]_{+}}<1$, hence, Theorem \ref{dfffftheorem3}
extends the results of   Ke and Zheng (\cite{Zhengddfggghjjkk1}) and  partly extends the result of Liu et al (\cite{Liughjj791}).

(ii) 
Obviously, if $\mu>0$,
$2-\frac{2}{N}\frac{\max_{s\geq1}\lambda_0^{\frac{1}{{{s}}+1}}(\chi+\xi\|w_0\|_{L^\infty(\Omega)})}{[\max_{s\geq1}\lambda_0^{\frac{1}{{{s}}+1}}(\chi+\xi\|w_0\|_{L^\infty(\Omega)})-\mu]_{+}}<2-\frac{2}{N}$, hence Theorem \ref{dfffftheorem3}
extends the results of   Wang (\cite{Wangscd331629}) and Li and Lankeit
 (\cite{Li445666}).

(iii) 
 Theorem \ref{dfffftheorem3}  extends the results of   Zheng et al.  (\cite{Zhengssddff0})
for the critical case $ \mu \geq \frac{(N-2)_+}{N}\chi$ as well as $w\equiv0$ and $D(u)=1.$

(iv) 
 If $\mu=0$ and $w\equiv0$, then \dref{1.1}
 possess some solutions which blow up in finite time provided that
$D$   satisfy \dref{9161}--\dref{9162} with $m<2-\frac{2}{N}$ (see e.g. \cite{Winkler79,Cie791}). Therefore,
In comparison to the result
for the corresponding haptotaxis-free system ($w\equiv0$ in \dref{1.1}), it is easy to see that the restriction on $m$ here is
{\bf optimal}. 

(v) If $N=2$, then $$
2-\frac{2}{N}\frac{\max_{s\geq1}
\lambda_0^{\frac{1}{{{s}}+1}}(\chi+\xi\|w_0\|_{L^\infty(\Omega)})}{[\max_{s\geq1}\lambda_0^{\frac{1}{{{s}}+1}}(\chi+\xi\|w_0\|_{L^\infty(\Omega)})-\mu]_{+}}=
2-\frac{\max_{s\geq1}
\lambda_0^{\frac{1}{{{s}}+1}}(\chi+\xi\|w_0\|_{L^\infty(\Omega)})}{[\max_{s\geq1}\lambda_0^{\frac{1}{{{s}}+1}}(\chi+\xi\|w_0\|_{L^\infty(\Omega)})-\mu]_{+}}<1,$$
therefore, Theorem \ref{theorem3}  extends the results of   Wang et al. (\cite{Zhenghhjjghjjkk1}),
who proved the possibility of global and bounded, in the cases, $D$   satisfies \dref{9161}--\dref{9162} with ${\bf m>1}$.

\end{remark}

The main novelty and difficulty of the paper is how to control the chemotaxis term $\chi\nabla\cdot(u\nabla v)$,  haptotaxis term $\xi\nabla\cdot(u\nabla w)$ and strong degeneracies caused
by system \dref{1.1}. To overcome this difficulty, the purpose of the present paper is to demonstrate   how far an adequate combination of maximal Sobolev regularity theory and
develop {\bf new} $L^p$-estimate techniques (see Lemmas \ref{lemma45630}--\ref{lemmaddffddfffrsedrffffffgg})   can be used to obtain the global existence and boundedness of solutions
to \dref{1.1}.  

The rest of the paper is organized in the following way.
 Section 2 will be concerned with preliminaries, including some basic facts and a local existence result. In section 3, by careful analysis,  this paper develops some $L^p$-estimate techniques to raise the a priori
estimate of a solution from $L^1 (\Omega) \rightarrow L^{\lambda-\varepsilon} (\Omega) \rightarrow L^{\lambda} (\Omega)\rightarrow
L^{\lambda+\varepsilon} (\Omega)\rightarrow L^{p} (\Omega)(~\mbox{for all}~~p>1)$,
where $$
 \lambda=\frac{\kappa_0}{(\kappa_0-\mu)_+}
$$
and
 \begin{equation}
 \kappa_0=\left\{\begin{array}{ll}
\max_{s\geq1}\lambda_0^{\frac{1}{{{s}}+1}}(\chi+\xi\|w_0\|_{L^\infty(\Omega)})~~\mbox{if}~~
~~ \tau=1,\\
\chi~~~~~~\mbox{if}~~ \tau=0.\\
 \end{array}\right.
\label{113344ddff4dffgg4zjscz2ddd.5297x96302222114}
\end{equation}
To this end,  by using the maximal Sobolev
regularity and the standard estimate for the solution, we may  derive entropy-like inequalities (see \dref{cz2sssssss.5xx1} and \dref{czssddssddffggf2.5kk1214114114rrggkkll}). Then in order to estimate
the right term
$\int_{\Omega}
u^{k+1}$ and $\int_{\Omega}
u^{k}$
 on the rightmost of \dref{cz2sssssss.5xx1} and \dref{czssddssddffggf2.5kk1214114114rrggkkll}, we need to  deal with for  two steps  from   $\|u(\cdot, t)\|_{L^{\lambda-\varepsilon}(\Omega)}\rightarrow\|u(\cdot, t)\|_{L^{\lambda}(\Omega)}$ (see the proof of Lemmas \ref{lemma45630} and \ref{qqqqlemma45630}), which plays a key rule in obtaining the main results. Then  employing a bootstrap
argument
 (see \dref{cz2.5sssssxsdddfffsdddx1} and \dref{cz2sfggg.5sssssxsdddsdddx1}), one could derive the boundedness of $\|u(\cdot, t)\|_{L^{\lambda+\varepsilon}(\Omega)}$ (see Lemma \ref{lessdddmma45630}).  Relying on this, we develop {\bf new} $L^p$-estimate techniques to raise the a priori estimate of solutions from
$L^{\lambda+\varepsilon}(\Omega) \rightarrow L^{p}(\Omega) (\mbox{for all}~p>1)$ (see Lemmas  \ref{lemmasdffssdd45566645630223}--\ref{lemmaddffddfffrsedrffffffgg}).
Finally,
 applying  the standard Alikakos--Moser iteration,  we prove the main results of this paper in the last part.

\section{Preliminaries}
In this section, we will recall  some lemmas and elementary inequalities which
will be used frequently later.
%
%
To begin with, let us collect some basic solution properties which essentially have already been used
in \cite{Keengwwwwssddghjjkk1}.
%

\begin{lemma}(\cite{Friedmangg791})\label{lemma41}
Let $\theta\in(0,p)$.
There exists a positive constant $C_{GN}$ such that for all $u \in W^{1,2}(\Omega)\cap L^\theta(\Omega)$,
$$\|u\|_{L^p(\Omega)} \leq C_{GN}(\|\nabla u\|_{L^2(\Omega)}^{a}\|u\|^{1-a}_{L^\theta(\Omega)}+\|u\|_{L^\theta(\Omega)})$$
is valid with 
$a =\disp{\frac{\frac{N}{\theta}-\frac{N}{p}}{1-\frac{N}{2}+\frac{N}{\theta}}}\in(0,1)$.
%
\end{lemma}


\begin{lemma}\label{lemma45xy1222232} (\cite{Keengwwwwssddghjjkk1})
Suppose  that $\gamma\in (1,+\infty)$ and $g\in L^\gamma((0, T); L^\gamma(
\Omega))$.
 Consider the following evolution equation
 $$
 \left\{\begin{array}{ll}
v_t -\Delta v+v=g,~~~(x, t)\in
 \Omega\times(0, T ),\\
\disp\frac{\partial v}{\partial \nu}=0,~~~(x, t)\in
 \partial\Omega\times(0, T ),\\
v(x,0)=v_0(x),~~~(x, t)\in
 \Omega.\\
 \end{array}\right.
 $$
 For each $v_0\in W^{2,\gamma}(\Omega)$
such that $\disp\frac{\partial v_0}{\partial \nu}=0$ and any $g\in L^\gamma((0, T); L^\gamma(
\Omega))$, there exists a unique solution
$v\in W^{1,\gamma}((0,T);L^\gamma(\Omega))\cap L^{\gamma}((0,T);W^{2,\gamma}(\Omega)).$ In addition, if $s_0\in[0,T)$, $v(\cdot,s_0)\in W^{2,\gamma}(\Omega)(\gamma>N)$ with $\disp\frac{\partial v(\cdot,s_0)}{\partial \nu}=0,$
then there exists a positive constant $\lambda_0:=\lambda_0(\Omega,\gamma,N)$ such that  
$$
\begin{array}{rl}
&\disp{\int_{s_0}^Te^{\gamma s}\| v(\cdot,t)\|^{\gamma}_{W^{2,\gamma}(\Omega)}ds\leq\lambda_0\left(\int_{s_0}^Te^{\gamma s}
\|g(\cdot,s)\|^{\gamma}_{L^{\gamma}(\Omega)}ds+e^{\gamma s_0}(\|v_0(\cdot,s_0)\|^{\gamma}_{W^{2,\gamma}(\Omega)})\right).}\\
\end{array}
$$
\end{lemma}

The local-in-time existence of classical solutions to the chemotaxis--haptotaxis model \dref{1.1} is
quite standard; see similar discussions in 
%
%
%
%
\cite{Tao72,Liughjj791}.
Therefore we omit it.
\begin{lemma}\label{lemma70}
Assume that the nonnegative functions $u_0,v_0,$ and $w_0$ satisfies \dref{x1.731426677gg} (or \dref{aassx1.731426677gg}, if $\tau=0$)
for some $\vartheta\in(0,1),$
$D$ satisfies \dref{9161} and \dref{9162}.
%
%
 Then there exists a maximal existence time $T_{max}\in(0,\infty]$ and a triple of  nonnegative functions 
 $$
\begin{array}{rl}
&\disp{u\in C^0(\bar{\Omega}\times[0,T_{max}))\cap C^{2,1}(\bar{\Omega}\times(0,T_{max})),}\\
&\disp{v\in C^0(\bar{\Omega}\times[0,T_{max}))\cap C^{2,1}(\bar{\Omega}\times(0,T_{max})),}\\
&\disp{w\in  C^{2,1}(\bar{\Omega}\times[0,T_{max}))}\\
\end{array}
$$
 which solves \dref{1.1}  classically and satisfies $w\leq \|w_0\|_{L^\infty(\Omega)}$
  in $\Omega\times(0,T_{max})$.
%
Moreover, if  $T_{max}<+\infty$, then
\begin{equation}
\|u(\cdot, t)\|_{L^\infty(\Omega)}+\|v(\cdot,t)\|_{W^{1,\infty}(\Omega)}\rightarrow\infty~~ \mbox{as}~~ t\nearrow T_{max}.
\label{1.163072x}
\end{equation}
\end{lemma}

According to the above existence theory,  for any $s\in(0, T_{max})$, $(u(\cdot, s), v(\cdot, s),w(\cdot, s))\in C^2(\bar{\Omega})$.
Without
 loss of generality, we can assume that there exists a positive constant 
$K$ such that
\begin{equation}\label{eqx45xx12112}
\|u_0\|_{C^2(\bar{\Omega})}+\|v_0\|_{C^2(\bar{\Omega})}+\|w_0\|_{C^2(\bar{\Omega})}\leq K.
\end{equation}


Employing  the same arguments as in the proof of Lemma 2.3 in \cite{Taoddfvbb41215} (see also \cite{Taox3201}), we derive  the following Lemma:
\begin{lemma}\label{lemm3a} Let $(u, v,w)$ solve \dref{1.1} in $\Omega\times(0, T_{max})$.
 Then
 \begin{equation}\label{x1.731426677gghh}
 \begin{array}{rl}
-\Delta w(x, t) \leq&\disp{\tau \|w_0\|_{L^\infty(\Omega)}\cdot v(x,t)+\kappa~~~\mbox{for all}~~x\in\Omega~~\mbox{and}~~~t\in(0, T_{max}),}\\
\end{array}
\end{equation}
where
\begin{equation}\label{x1.73ddff1426677gghh}
\kappa:=\|\Delta w_0\|_{L^\infty(\Omega)}+4\|\nabla\sqrt{w_0}\|_{L^\infty(\Omega)}^2+\frac{\|w_0\|_{L^\infty(\Omega)}}{e}.
\end{equation}

\end{lemma}

\section{A priori estimates}

In this section, we are going to establish an iteration step to develop the main ingredient of our result.
The iteration depends on a series of a priori estimates.
Firstly, 
%
%
%
the following two
lemmas provide some elementary material that will be essential to our bootstrap
procedure.

\begin{lemma}\label{333ssdeedrfe116lemma70hhjj}
Let $\mu=0,$ 
then
 the solution $(u, v,w)$ of  \dref{1.1} satisfies
%
%
\begin{equation}
 \begin{array}{rl}
 \|u(\cdot,t)\|_{L^1(\Omega)}= \|u_0\|_{L^1(\Omega)}~~~\mbox{for all}~~t\in (0, T_{max}).
\end{array}\label{333ssddaqwswddaassffssff3.ddfvbb10deerfgghhjuuloollgghhhyhh}
\end{equation}

%
\end{lemma}

In contrast to the situation without source terms ($\mu=0$ in \dref{1.1}), we cannot hope for mass conservation in the first
component. Nevertheless, the following inequality still holds:
\begin{lemma}\label{ssdeedrfe116lemma70hhjj} (see e.g. \cite{Zhengddfggghjjkk1})
Assume that $\mu>0.$
There exists a positive constant  
$ K_0$ such that
 the solution $(u, v,w)$ of  \dref{1.1} satisfies
%
%
\begin{equation}
 \begin{array}{rl}
 \|u(\cdot,t)\|_{L^1(\Omega)}\leq K_0~~~\mbox{for all}~~t\in (0, T_{max})
\end{array}\label{ssddaqwswddaassffssff3.ddfvbb10deerfgghhjuuloollgghhhyhh}
\end{equation}
and
\begin{equation}
\int_t^{t+\tau}\int_{\Omega}{u^{2}}\leq  K_0~~\mbox{for all}~~ t\in(0, T_{max}-\tau),
\label{bnmbncz2.5ghhjuyuivvddfggghhbssdddeennihjj}
\end{equation}
where \begin{equation}
\tau:=\min\{1,\frac{1}{6}T_{max}\}.
\label{cz2.5ghju48cfg924vbhu}
\end{equation}

%
\end{lemma}

Now, we now proceed to derive a uniform  upper bound for $u$, which
turns out to be the key to obtain all the higher order estimates and thus to extend the
classical solution globally.
%
%
To do this, employing  the maximal Sobolev
regularity, in light of Lemma \ref{333ssdeedrfe116lemma70hhjj},  as a first conclusion  towards global existence of the classical
 solutions is the following a priori
estimate which asserts that, in sharp contrast to the case $\mu=0$ (see also \cite{Winkler37103})
 is a priori
uniformly bounded in $L^k(\Omega)$ for some $k$ larger than one. In order to deal with  the critical case ($k=\lambda$),  the novelty of paper, we first obtain the bounded of $\|u(\cdot, t)\|_{L^{k_0}(\Omega)}$
where
$k_0\in (\max\{1,\lambda-\frac{N}{2}\},\lambda)$. And then by some careful analysis, one can finally derive the bounded of the critical case, which are the following Lemmas:
%
%
%
%
%
\begin{lemma}\label{3455667lemma45630}
Let $(u,v,w)$ be a solution to \dref{1.1} on $(0,T_{max})$. 
 Then for any $k>1,$
 one can
find positive constants  $\rho_0$ and $\rho_1$ such that
\begin{equation}
\begin{array}{rl}
&\disp{\frac{1}{{k}}\|u(\cdot,t) \|^{{{k}}}_{L^{{k}}(\Omega)}+C_ D(k-1)\int_{0}^t
e^{-( { {k}+1})(t-s)}\int_{\Omega}u^{m+k-3}|\nabla u|^2dxds}
\\
\leq&\disp{[\frac{({k}-1)}{{k}}\kappa_0- \mu]\int_{0}^t
e^{-( {k}+1)(t-s)}\int_\Omega u^{{{k}+1}}dxds }\\
&+\disp{ \rho_0\int_{0}^t
e^{-( { {k}+1})(t-s)}\int_\Omega u^{{{k}}} dxds+\rho_1~~~\mbox{for all}~~t\in (0, T_{max}),}\\
\end{array}
\label{czssddssddffggf2.3456675kk1214114114rrggkkll}
\end{equation}
where $\kappa_0$ is the same as \dref{113344ddff4dffgg4zjscz2ddd.5297x96302222114}.
\end{lemma}
\begin{proof}
Multiplying the first equation in  \dref{1.1}
  by $u^{k-1}$, and integrating in space  and using $w\geq0$, 
 we get
\begin{equation}
\begin{array}{rl}
&\disp{\frac{1}{k}\frac{d}{dt}\|u\|^{k}_{L^k(\Omega)}+(k-1)\int_{\Omega}u^{k-2} D(u)|\nabla u|^2dx}
\\
\leq&\disp{-\chi\int_\Omega \nabla\cdot(u\nabla v)u^{k-1}dx-\xi\int_\Omega\nabla\cdot
  (u\nabla w)u^{k-1}+\mu
\int_\Omega  u^{k-1} u(1- u-w) }\\
\leq&\disp{-\chi\int_\Omega \nabla\cdot(u\nabla v)u^{k-1}dx-\xi\int_\Omega\nabla\cdot
  (u\nabla w)u^{k-1}+\mu
\int_\Omega  u^{k-1} u(1- u) }\\
\end{array}
\label{cz2.5xx1jjjj}
\end{equation}
for all $t\in(0,T_{max})$.

{\bf Case  $\tau=0:$}
Integrating by parts to the first term on the right hand side of \dref{cz2.5xx1jjjj} and %
from
$\dref{1.1}_2$ we obtain
\begin{equation}
\begin{array}{rl}
&\disp{-\chi\int_\Omega \nabla\cdot(u\nabla v)u^{k-1}}
\\
=&\disp{(k-1 )\chi\int_\Omega  u^{k-1}\nabla u\cdot\nabla v }
\\
\leq&\disp{\frac{k-1}{k }\chi  \int_\Omega u^{k+1},}\\
\end{array}
\label{cz2.5630111}
\end{equation}
where  we have used the fact that $v\geq0.$
Summing up \dref{x1.731426677gghh} and \dref{x1.73ddff1426677gghh}
yields to 
\begin{equation}
\begin{array}{rl}
&\disp{-\xi\int_\Omega \nabla\cdot( u \nabla w)
 u^{k-1} }
\\
=&\disp{-\frac{({k-1})}{k}\xi\int_\Omega u^{k}\Delta w }
\\
\leq&\disp{\kappa\frac{({k-1})\xi}{k}\int_\Omega u^{k}}\\
\leq&\disp{\kappa\xi\int_\Omega u^{k},}\\
\end{array}
\label{cz2.563019rrtttt12}
\end{equation}
where
$\kappa$ is give by \dref{x1.73ddff1426677gghh} and
we have used the fact that $\tau=0$ in \dref{x1.731426677gghh}.
Here,  by some basic calculation, we deduce that
\begin{equation}
\begin{array}{rl}
&\disp{\mu
\int_\Omega  u^{k-1} u(1- u) =-\mu\int_\Omega   u^{k+1}+\mu \int_\Omega   u^{k}.}\\
\end{array}
\label{cz2.56301hh}
\end{equation}
Therefore, combined with
\dref{cz2.563019rrtttt12}, \dref{cz2.56301hh},
and \dref{cz2.5xx1jjjj}  and \dref{9162}, we have
\begin{equation}
\begin{array}{rl}
&\disp{\frac{1}{k}\frac{d}{dt}\|u\|^{k}_{L^k(\Omega)}+\frac{4C_ D(k-1)}{(m+k-1)^2}\|\nabla u^{\frac{m+k-1}{2}}\|^2_{L^2(\Omega)}+\frac{k+1}{k}\int_\Omega  u^{k}}
\\
\leq&\disp{(-\mu+\frac{k-1}{k }\chi) \int_\Omega u^{k+1}+C_1\int_\Omega   u^{k}~~\mbox{for all}~~ t\in(0,T_{max}),}\\
\end{array}
\label{cz2sssssss.5xx1}
\end{equation}
where
 \begin{equation}
 C_1=\mu+\kappa\xi+2.
 \label{cz2sssssssdertts.5xx1}
\end{equation} Here we have used the fact that $\frac{k+1}{k}\leq2$ (by $k>1$).
For any $t\in (0,T_{max})$, applying the Gronwall Lemma   to the above inequality, we have
\begin{equation}
\begin{array}{rl}
&\disp{\frac{1}{{k}}\|u(\cdot,t) \|^{{{k}}}_{L^{{k}}(\Omega)}+C_ D(k-1)\int_{0}^t
e^{-( { {k}+1})(t-s)}\int_{\Omega}u^{m+k-3}|\nabla u|^2dxds}
\\
\leq&\disp{(-\mu+\frac{k-1}{k }\chi) \int_{0}^t\int_\Omega u^{k+1}dxds+C_1\int_{0}^t\int_\Omega   u^{k}dxds+C_2~~\mbox{for all}~~ t\in(0,T_{max}),}\\
\end{array}
\label{cz2ssssssdffffggsss.5jjjxx1}
\end{equation}
with
\begin{equation}C_2:=C_2({k})=\frac{1}{{k}}\|u_0\|^{{{k}}}_{L^{{k}}(\Omega)}.
 \label{cz2ssssssdffffggsss.5sethjkkjjjxx1}
\end{equation}


{\bf Case  $\tau=1:$}
Integrating by parts to the first term on the right hand side of \dref{cz2.5xx1jjjj} and %
from
$\dref{1.1}_2$ we obtain for any $\varepsilon_1>0,$
\begin{equation}
\begin{array}{rl}
&\disp{-\chi\int_\Omega \nabla\cdot(u\nabla v)u^{k-1}}
\\
=&\disp{-\frac{({k-1})\chi}{k}\int_\Omega u^{k}\Delta v }
\\
\leq&\disp{\frac{(k-1)\chi}{k}\int_\Omega  u^{k}|\Delta v| }
\\
\leq&\disp{\varepsilon_1\int_\Omega  u^{k+1}+\gamma_1\varepsilon_1^{-k}\int_\Omega|\Delta v|^{k+1}, }
\\
\end{array}
\label{223444cz2.5630111}
\end{equation}
where  \begin{equation}\gamma_1=\frac{1}{k+1}\left(\frac{k+1}{k}\right)^{-k}\left(\frac{(k-1)\chi}{k}\right)^{k+1}.
\label{22ddf34ddff44cz2.5630111}
\end{equation}

Due to \dref{x1.731426677gghh} and \dref{x1.73ddff1426677gghh}, it follows that for any $\varepsilon_2>0,$
\begin{equation}
\begin{array}{rl}
&\disp{-\xi\int_\Omega \nabla\cdot( u \nabla w)
 u^{k-1} }
\\
=&\disp{-\frac{({k-1})\xi}{k}\int_\Omega u^{k}\Delta w }
\\
\leq&\disp{\kappa\frac{({k-1})\xi}{k}\int_\Omega u^{k}+\frac{({k-1})\xi\|w_0\|_{L^\infty(\Omega)}}{k}\int_\Omega u^{k}v}\\
\leq&\disp{\kappa\xi\int_\Omega u^{k}+\frac{({k-1})\xi\|w_0\|_{L^\infty(\Omega)}}{k}\int_\Omega u^{k}v}\\
\leq&\disp{\kappa\xi\int_\Omega u^{k}+\frac{({k-1})\xi\|w_0\|_{L^\infty(\Omega)}}{k}\int_\Omega u^{k}v}\\
\leq&\disp{\kappa\xi\int_\Omega u^{k}+\varepsilon_2\int_\Omega u^{k+1}+\gamma_2\varepsilon_2^{-k}\int_{\Omega}v^{k+1},}\\
\end{array}
\label{qqqqcz2.563019rrtttt12}
\end{equation}
where
\begin{equation}\gamma_2:=\frac{1}{k+1}\left(\frac{k+1}{k}\right)^{-k}
\left(\frac{({k-1})\xi\|w_0\|_{L^\infty(\Omega)}}{k}\right)^{k+1}.
\label{qqqssffffqcz2.563019rrtttt12}
\end{equation}
Here $\kappa$ is give by \dref{x1.73ddff1426677gghh} and
we have used the fact that $\tau=1$ in \dref{x1.731426677gghh}.
On the other hand, in view of the Young inequality, we also derive that
\begin{equation}
\begin{array}{rl}
\mu
\disp\int_\Omega  u^{k-1} u(1- u) =&\disp{-\mu\int_\Omega   u^{k+1}+(\mu+\frac{k+1}{k})\int_\Omega  u^{k}-\frac{k+1}{k}\int_\Omega  u^{k}}\\
\leq&\disp{-\mu\int_\Omega   u^{k+1}+(\mu+2)\int_\Omega  u^{k}-\frac{k+1}{k}\int_\Omega  u^{k}}\\
\end{array}
\label{sddddcz2.56301hh}
\end{equation}
by using $k>1.$
Therefore, combined with \dref{223444cz2.5630111},
\dref{qqqqcz2.563019rrtttt12}, \dref{cz2.5xx1jjjj} as well as 
 \dref{sddddcz2.56301hh}  and \dref{9162}, we have
\begin{equation}
\begin{array}{rl}
&\disp{\frac{1}{k}\frac{d}{dt}\|u\|^{k}_{L^k(\Omega)}+C_ D(k-1)\int_{\Omega}u^{m+k-3}|\nabla u|^2+\frac{k+1}{k}\int_\Omega  u^{k}}
\\
\leq&\disp{(-\mu +\varepsilon_1+\varepsilon_2 )\int_\Omega u^{k+1}+\gamma_1\int_\Omega|\Delta v|^{k+1}+\gamma_2\int_\Omega v^{k+1}+C_1\int_\Omega  u^{k},}\\
\end{array}
\label{cz2.5xx1}
\end{equation}
where  $C_1$ is the same as \dref{cz2sssssssdertts.5xx1}.
For any $t\in (0,T_{max})$, applying the Gronwall Lemma   to the above inequality shows that
\begin{equation}
\begin{array}{rl}
&\disp{\frac{1}{{k}}\|u(\cdot,t) \|^{{{k}}}_{L^{{k}}(\Omega)}+C_ D(k-1)\int_{0}^t
e^{-( { {k}+1})(t-s)}\int_{\Omega}u^{m+k-3}|\nabla u|^2}
\\
\leq&\disp{\frac{1}{{k}}e^{-( { {k}+1})t}\|u_0 \|^{{{k}}}_{L^{{k}}(\Omega)}+(\varepsilon_1+\varepsilon_2- \mu)\int_{0}^t
e^{-( { {k}+1})(t-s)}\int_\Omega u^{{{k}+1}} dxds}\\
&+\disp{\gamma_1\varepsilon_1^{-k}\int_{0}^t
e^{-( { {k}+1})(t-s)}\int_\Omega |\Delta v|^{ {k}+1} dxds+ C_1\int_{0}^t
e^{-( { {k}+1})(t-s)}\int_\Omega u^{{{k}}} dxds}\\
&\disp{+\gamma_2\varepsilon_2^{-k}\int_{0}^t
e^{-( { {k}+1})(t-s)}\int_\Omega v^{{{k}+1}} dxds}\\
\leq&\disp{(\varepsilon_1+\varepsilon_2- \mu)\int_{0}^t
e^{-( { {k}+1})(t-s)}\int_\Omega u^{{{k}+1}} dxds+\gamma_1\varepsilon_1^{-k}\int_{0}^t
e^{-( { {k}+1})(t-s)}\int_\Omega |\Delta v|^{ {k}+1} dxds}\\
&+\disp{\gamma_2\varepsilon_2^{-k}\int_{0}^t
e^{-( { {k}+1})(t-s)}\int_\Omega v^{{{k}+1}} dxds+
 C_1\int_{0}^t
e^{-( { {k}+1})(t-s)}\int_\Omega u^{{{k}}} dxds+C_2,}\\
\end{array}
\label{cz2111ddffdfghhhg11.5kk1214114114rrgg}
\end{equation}
where
$C_2$ is given by \dref{cz2ssssssdffffggsss.5sethjkkjjjxx1}.
Next, a use of Lemma \ref{lemma45xy1222232} leads to
\begin{equation}\label{cz2.5kke3456778999ddff9001214114114rrggjjkk}
\begin{array}{rl}
&\disp{\gamma_1\varepsilon_1^{-k}\int_{0}^t
e^{-( { {k}+1})(t-s)}\int_\Omega |\Delta v|^{ {k}+1} dxds}
\\
=&\disp{\gamma_1\varepsilon_1^{-k}e^{-( { {k}+1})t}\int_{0}^t
e^{( { {k}+1})s}\int_\Omega |\Delta v|^{ {k}+1} dxds}\\
\leq&\disp{\gamma_1\varepsilon_1^{-k}e^{-( { {k}+1})t}\lambda_0(\int_{0}^t
\int_\Omega e^{( { {k}+1})s}u^{ {k}+1} dxds+\|v_0\|^{ {k}+1}_{W^{2, { {k}+1}}})}\\
\end{array}
\end{equation}
and
\begin{equation}\label{cz2.5kk12141141dfggghhh14rrggjjkk}
\begin{array}{rl}
&\disp{\gamma_2\varepsilon_2^{-k}\int_{0}^t
e^{-( { {k}+1})(t-s)}\int_\Omega v^{ {k}+1} dxds}
\\
=&\disp{\gamma_2\varepsilon_2^{-k}e^{-( { {k}+1})t}\int_{0}^t
e^{( { {k}+1})s}\int_\Omega  v^{ {k}+1} dxds}\\
\leq&\disp{\gamma_2\varepsilon_2^{-k}e^{-( { {k}+1})t}\lambda_0(\int_{0}^t
\int_\Omega e^{( { {k}+1})s}u^{ {k}+1} dxds+\|v_0\|^{ {k}+1}_{W^{2, { {k}+1}}})}\\
\end{array}
\end{equation}
for all $t\in(0, T_{max})$.
On the other hand, choosing $\varepsilon_1=\frac{(k-1)\chi}{k+1}\lambda_0^{\frac{1}{k+1}}$ and
 $\varepsilon_2=\frac{(k-1)\xi\|w_0\|_{L^\infty(\Omega)}}{k+1}\lambda_0^{\frac{1}{k+1}}$, with the help of \dref{22ddf34ddff44cz2.5630111} and \dref{qqqssffffqcz2.563019rrtttt12},  a simple calculation  shows that
$$\varepsilon_1+\gamma_1\lambda_0\varepsilon_1^{-k}=\frac{({k}-1)}{{k}}\lambda_0^{\frac{1}{{k}+1}}\chi$$
and
$$\varepsilon_2+\gamma_2\lambda_0\varepsilon_2^{-k}=\frac{({k}-1)}{{k}}\lambda_0^{\frac{1}{{k}+1}}\xi\|w_0\|_{L^\infty(\Omega)},$$
so that,
substituting \dref{cz2.5kke3456778999ddff9001214114114rrggjjkk}--\dref{cz2.5kk12141141dfggghhh14rrggjjkk} into \dref{cz2111ddffdfghhhg11.5kk1214114114rrgg} implies that 
\begin{equation}
\begin{array}{rl}
&\disp{\frac{1}{{k}}\|u(\cdot,t) \|^{{{k}}}_{L^{{k}}(\Omega)}+C_ D(k-1)\int_{0}^t
e^{-( { {k}+1})(t-s)}\int_{\Omega}u^{m+k-3}|\nabla u|^2dxds}
\\
\leq&\disp{(\varepsilon_1+\gamma_1\lambda_0\varepsilon_1^{-k}+\varepsilon_2+\gamma_2\lambda_0\varepsilon_2^{-k}- \mu)\int_{0}^t
e^{-( {k}+1)(t-s)}\int_\Omega u^{{{k}+1}} dxds}\\
&+\disp{(\gamma_1\varepsilon_2^{-k}+\gamma_2\varepsilon_2^{-k})e^{-( {k}+1) t }\lambda_0\|v_0\|^{ {k}+1}_{W^{2, { {k}+1}}}+
 C_1\int_{0}^t
e^{-( { {k}+1})(t-s)}\int_\Omega u^{{{k}}} dxds+C_2}\\
=&\disp{(\frac{({k}-1)}{{k}}\lambda_0^{\frac{1}{{k}+1}}\chi+\frac{({k}-1)}{{k}}\lambda_0^{\frac{1}{{k}+1}}\xi\|w_0\|_{L^\infty(\Omega)}- \mu)\int_{0}^t
e^{-( { {k}+1})(t-s)}\int_\Omega u^{{{k}+1}} dxds}\\
&+\disp{(\gamma_1\varepsilon_2^{-k}+\gamma_2\varepsilon_2^{-k})e^{-( {k}+1) t }\lambda_0\|v_0\|^{ {k}+1}_{W^{2, { {k}+1}}}+
 C_1\int_{0}^t
e^{-( { {k}+1})(t-s)}\int_\Omega u^{{{k}}} dxds+C_2}\\
\leq&\disp{[\frac{({k}-1)}{{k}}\max_{s\geq1}\lambda_0^{\frac{1}{{{s}}+1}}(\chi+\xi\|w_0\|_{L^\infty(\Omega)})- \mu]\int_{0}^t
e^{-( {k}+1)(t-s)}\int_\Omega u^{{{k}+1}}dxds }\\
&+\disp{ C_1\int_{0}^t
e^{-( { {k}+1})(t-s)}\int_\Omega u^{{{k}}} dxds+C_3}\\
\end{array}
\label{czssddssddffggf2.5kk1214114114rrggkkll}
\end{equation}
with $$C_3=(\gamma_1\varepsilon_2^{-k}+\gamma_2\varepsilon_2^{-k})e^{-( {k}+1) t }\lambda_0\|v_0\|^{ {k}+1}_{W^{2, { {k}+1}}}+C_2.$$
Finally,  choosing $\rho_0=\mu+\kappa\xi+2$ and $\rho_1=C_3$, using \dref{cz2ssssssdffffggsss.5jjjxx1} and \dref{czssddssddffggf2.5kk1214114114rrggkkll}, applying \dref{113344ddff4dffgg4zjscz2ddd.5297x96302222114}, we derive that \dref{czssddssddffggf2.3456675kk1214114114rrggkkll} holds.
\end{proof}

\begin{lemma}\label{lemma45630}
Let $(u,v,w)$ be a solution to \dref{1.1} on $(0,T_{max})$. If $\tau=0$ and $\mu>0,$
then 
 for any
\begin{equation}
k\in\left\{\begin{array}{ll}
(1,\frac{\chi}{(\chi-\mu)_{+}}],~~\mbox{if}~~
0<\mu<\chi~~{and}~~m\geq2-\frac{2}{N}\frac{\chi}{(\chi-\mu)_{+}},\\
(1,\frac{\chi}{(\chi-\mu)_{+}}),~~\mbox{if}~~
\mu\geq\chi,\\
 \end{array}\right.
\label{zjscz2.52ddfff97x9630111}
\end{equation}
 one can
find a positive constant $C$ such that
\begin{equation}
\|u(\cdot, t)\|_{L^k(\Omega)}\leq C ~~\mbox{for all}~~ t\in(0,T_{max})
\label{zjscz2.5297x9630111}
\end{equation}
holds.
\end{lemma}
\begin{proof}
Firstly, applying Lemma \ref{3455667lemma45630}, using $\tau=0$   and \dref{113344ddff4dffgg4zjscz2ddd.5297x96302222114}, we conclude that
\begin{equation}
\begin{array}{rl}
&\disp{\frac{1}{{k}}\|u(\cdot,t) \|^{{{k}}}_{L^{{k}}(\Omega)}+C_ D(k-1)\int_{0}^t
e^{-( { {k}+1})(t-s)}\int_{\Omega}u^{m+k-3}|\nabla u|^2dxds}
\\
\leq&\disp{[\frac{({k}-1)}{{k}}\chi- \mu]\int_{0}^t
e^{-( {k}+1)(t-s)}\int_\Omega u^{{{k}+1}}dxds }\\
&+\disp{ \rho_0\int_{0}^t
e^{-( { {k}+1})(t-s)}\int_\Omega u^{{{k}}} dxds+\rho_1~~~\mbox{for all}~~t\in (0, T_{max}),}\\
\end{array}
\label{cz2sssssss.5xx1}
\end{equation}
where $\rho_0$ and  $\rho_0$ are the same as  Lemma \ref{3455667lemma45630}.

{\bf Case $\mu<\chi$}: 

{\bf Step 1. The boundedness of $\|u(\cdot, t)\|_{L^{k}(\Omega)}$
for all $t\in (0, T_{max})$ and $k\in (1,\frac{\chi}{(\chi-\mu)_{+}})$.
}

  To this end, for any $\varepsilon>0$, pick  $k=\frac{\chi}{(\chi-\mu)_{+}}-\varepsilon$ in  \dref{cz2sssssss.5xx1}, then, $-\mu+\frac{k-1}{k }\chi<0$ (by $0<\mu<\chi$), so that, \dref{cz2sssssss.5xx1} implies that there exists a positive constant $C_1$ such that
\begin{equation}
\begin{array}{rl}
&\disp{\frac{1}{{k}}\|u(\cdot,t) \|^{{{k}}}_{L^{{k}}(\Omega)}+C_ D(k-1)\int_{0}^t
e^{-( { {k}+1})(t-s)}\int_{\Omega}u^{m+k-3}|\nabla u|^2dxds }
\\
&+\disp{\frac{1}{2}[\mu-\frac{({k}-1)}{{k}}\chi]\int_{0}^t
e^{-( {k}+1)(t-s)}\int_\Omega u^{{{k}+1}}dxds }\\
\leq&\disp{C_1~~~\mbox{for all}~~t\in (0, T_{max})}\\
\end{array}
\label{cz2.5sssssxx1}
\end{equation}
by using the Young inequality.
Applying the Gronwall lemma to  \dref{cz2.5sssssxx1}, we derive
\begin{equation}
\|u(\cdot, t)\|_{L^{\frac{\chi}{(\chi-\mu)_{+}}-\varepsilon}(\Omega)}\leq C_2~~\mbox{for all}~~ t\in(0,T_{max}),
\label{zjscz2.5297xssffggg96ssdddd30111}
\end{equation}
which combined with  the arbitrariness of $\varepsilon$ and the H\"{o}lder  inequality yields to for any $k\in (1,\frac{\chi}{(\chi-\mu)_{+}})$,
\begin{equation}
\|u(\cdot, t)\|_{L^{k}(\Omega)}\leq C_3~~\mbox{for all}~~ t\in(0,T_{max}).
\label{zjscz2.5297xssss96ssdddd30111}
\end{equation}

{\bf Step 2. The boundedness of $\|u(\cdot, t)\|_{L^{k}(\Omega)}$
for all $t\in (0, T_{max})$ and $k\in (1,\frac{\chi}{(\chi-\mu)_{+}}]$.}

To achieve this, by step 1,
we may pick $k_0\in (\max\{1,\frac{\chi}{(\chi-\mu)_{+}}-\frac{N}{2}\},\frac{\chi}{(\chi-\mu)_{+}})$ such that
\begin{equation}
\|u(\cdot, t)\|_{L^{k_0}(\Omega)}\leq C_4~~\mbox{for all}~~ t\in(0,T_{max}).
\label{zjscz2.5297x96ssdddd30111}
\end{equation}
Now, set
$k=\frac{\chi}{(\chi-\mu)_{+}}$ in  \dref{cz2sssssss.5xx1}, then, $-\mu+\frac{k-1}{k }\chi=0$ (by $0<\mu<\chi$), so that, \dref{cz2sssssss.5xx1} implies that
\begin{equation}
\begin{array}{rl}
&\disp{\frac{1}{{k}}\|u(\cdot,t) \|^{{{k}}}_{L^{{k}}(\Omega)}+C_ D(k-1)\int_{0}^t
e^{-( { {k}+1})(t-s)}\int_{\Omega}u^{m+k-3}|\nabla u|^2dxds}
\\
\leq&\disp{ \rho_0\int_{0}^t
e^{-( { {k}+1})(t-s)}\int_\Omega u^{{{k}}} dxds+\rho_1~~~\mbox{for all}~~t\in (0, T_{max}).}\\
\end{array}
\label{cz2ssssse45677ss.5ssdddxx1}
\end{equation}
Now, observe that $m\geq2-\frac{2}{N}\frac{\chi}{(\chi-\mu)_{+}}$ and $k_0>\max\{1,\frac{\chi}{(\chi-\mu)_{+}}-\frac{N}{2}\}$ implies that
$$
\begin{array}{rl}
\disp{m+k-1+\frac{2}{N}\times k_0}&\geq{2-\frac{2}{N}\frac{\chi}{(\chi-\mu)_{+}}+k-1+\frac{2}{N}\times k_0}\\
&={k+1-\frac{2}{N}\frac{\chi}{(\chi-\mu)_{+}}+\frac{2}{N}\times k_0}\\
&>{k+1-\frac{2}{N}\frac{\chi}{(\chi-\mu)_{+}}+\frac{2}{N}\times (\frac{\chi}{(\chi-\mu)_{+}}-\frac{N}{2})}\\
&={k,}\\
\end{array}
$$
therefore, in view of \dref{zjscz2.5297x96ssdddd30111}, a use of the Gagliardo-Nirenberg inequality to \dref{cz2ssssse45677ss.5ssdddxx1} implies
 that
  there exist positive constants  $C_{5}$ and $C_6$ such that
\begin{equation}
\begin{array}{rl}
&\disp \rho_0\int_{\Omega} u^{k}\\
=&\disp{\|  u^{\frac{m+k-1}{2}}\|^{\frac{2{k}}{m+k-1}}_{L^{\frac{2{k}}{m+k-1}}(\Omega)}}\\
\leq&\disp{C_{5}(\|\nabla    u^{\frac{m+k-1}{2}}\|_{L^2(\Omega)}^{\frac{m+k-1}{k}\frac{N(k-k_0)}{N(m+k-1)+(2-N)k_0}}\|   u^{\frac{m+k-1}{2}}\|_{L^\frac{2k_0}{m+k-1 }(\Omega)}^{1-\frac{m+k-1}{k}\frac{N(k-k_0)}{N(m+k-2)+(2-N)k_0}}}\\
&\disp{+\|   u^{\frac{m+k-1}{2}}\|_{L^\frac{2k_0}{m+k-1 }(\Omega)})^{\frac{2{k}}{m+k-1}}}\\
\leq&\disp{C_{6}(\|\nabla    u^{\frac{m+k-1}{2}}\|_{L^{2}(\Omega)}^{\frac{2N(k-k_0)}{N(m+k-1)+(2-N)k_0}}+1),}\\
\end{array}
\label{123cz2.57151hhdfkkhjdsssdffgukildrftjj}
\end{equation}
where combined with ${\frac{2N(k-k_0)}{N(m+k-1)+(2-N)k_0}}<2$ (by $m+k-1+\frac{2}{N}\times k_0>k$) implies that
$$
\begin{array}{rl}
\disp \rho_0\int_{\Omega} u^{k}\leq\disp{\frac{2C_ D(k-1)}{(m+k-1)^2}\|\nabla    u^{\frac{m+k-1}{2}}\|_{L^{2}(\Omega)}^{2}+C_7,}\\
\end{array}
$$
for some positive constant $C_7$.
Substituting the above inequality into \dref{cz2ssssse45677ss.5ssdddxx1},  one can easily deduce that there exists a positive constant $C_8$ such that
\begin{equation}
\begin{array}{rl}
\|u(\cdot, t)\|_{L^{k_0}(\Omega)}\leq C_8~~\mbox{for all}~~ t\in(0,T_{max}).
\end{array}
\label{2.sssssssx1yy}
\end{equation}

{\bf Case $\mu\geq\chi:$} In view of $1<k<\frac{\chi}{(\chi-\mu)_{+}},$
$$-\mu+\frac{k-1}{k }\chi<0,$$
%
%
%
so that, \dref{cz2sssssss.5xx1} and  the Young inequality yields to
$$
\begin{array}{rl}
&\disp{\frac{1}{{k}}\|u(\cdot,t) \|^{{{k}}}_{L^{{k}}(\Omega)}+C_ D(k-1)\int_{0}^t
e^{-( { {k}+1})(t-s)}\int_{\Omega}u^{m+k-3}|\nabla u|^2dxds }
\\
&+\disp{\frac{1}{2}[\mu-\frac{({k}-1)}{{k}}\chi]\int_{0}^t
e^{-( {k}+1)(t-s)}\int_\Omega u^{{{k}+1}}dxds }\\
\leq&\disp{C_9~~~\mbox{for all}~~t\in (0, T_{max}).}\\
\end{array}
$$
 This completes the proof.
\end{proof}

\begin{lemma}\label{qqqqlemma45630}
Let $(u,v,w)$ be a solution to \dref{1.1} on $(0,T_{max})$ and $\theta_0=\frac{\max_{s\geq1}\lambda_0^{\frac{1}{{{s}}+1}}(\chi+\xi\|w_0\|_{L^\infty(\Omega)})}{\left[\max_{s\geq1}
 \lambda_0^{\frac{1}{{{s}}+1}}(\chi+\xi\|w_0\|_{L^\infty(\Omega)})-\mu\right]_{+}}$. If $\tau=1$ and $\mu>0,$
  then 
 for any
 \begin{equation}
k\in\left\{\begin{array}{ll}
(1,\theta_0],~~\mbox{if}~~
0<\mu<\max_{s\geq1}
 \lambda_0^{\frac{1}{{{s}}+1}}(\chi+\xi\|w_0\|_{L^\infty(\Omega)})~~{and}~~m\geq2-\frac{2}{N}\theta_0,\\
(1,\theta_0),~~\mbox{if}~~
\mu\geq\max_{s\geq1}\lambda_0^{\frac{1}{{{s}}+1}}(\chi+\xi\|w_0\|_{L^\infty(\Omega)}),\\
 \end{array}\right.
\label{zjscz2.52ddfff97xsss9630111}
\end{equation}
there exists a positive constant $C$ which depends on  $k$ such that
\begin{equation}
\|u(\cdot, t)\|_{L^k(\Omega)}\leq C ~~\mbox{for all}~~ t\in(0,T_{max})
\label{qqqqzjscz2.5297x9630111}
\end{equation}
holds.
\end{lemma}
\begin{proof}
Firstly, due to Lemma \ref{3455667lemma45630} as well as $\tau=1$   and \dref{113344ddff4dffgg4zjscz2ddd.5297x96302222114}, we have
\begin{equation}
\begin{array}{rl}
&\disp{\frac{1}{{k}}\|u(\cdot,t) \|^{{{k}}}_{L^{{k}}(\Omega)}+C_ D(k-1)\int_{0}^t
e^{-( { {k}+1})(t-s)}\int_{\Omega}u^{m+k-3}|\nabla u|^2dxds}
\\
\leq&\disp{[\frac{({k}-1)}{{k}}\max_{s\geq1}
 \lambda_0^{\frac{1}{{{s}}+1}}(\chi+\xi\|w_0\|_{L^\infty(\Omega)})- \mu]\int_{0}^t
e^{-( {k}+1)(t-s)}\int_\Omega u^{{{k}+1}}dxds }\\
&+\disp{ \rho_0\int_{0}^t
e^{-( { {k}+1})(t-s)}\int_\Omega u^{{{k}}} dxds+\rho_1~~~\mbox{for all}~~t\in (0, T_{max}),}\\
\end{array}
\label{5677888cz2sssssss.5xx1}
\end{equation}
where $\rho_0$ and  $\rho_0$ are the same as  Lemma \ref{3455667lemma45630}.
In the sequel, we wish to bound the  terms on the right-hand side of \dref{5677888cz2sssssss.5xx1}
in terms of the dissipation term on its left-hand side.

{\bf Case $\mu<\max_{s\geq1}
 \lambda_0^{\frac{1}{{{s}}+1}}(\chi+\xi\|w_0\|_{L^\infty(\Omega)})$:}
 
 {\bf Step 1. The boundedness of $\|u(\cdot, t)\|_{L^{k_0}(\Omega)}$
for all $t\in (0, T_{max})$ and $k_0\in (\max\{1,\theta_0-\frac{N}{2}\},\theta_0)$ with $\theta_0=\frac{\max_{s\geq1}\lambda_0^{\frac{1}{{{s}}+1}}(\chi+\xi\|w_0\|_{L^\infty(\Omega)})}{\left[\max_{s\geq1}
 \lambda_0^{\frac{1}{{{s}}+1}}(\chi+\xi\|w_0\|_{L^\infty(\Omega)})-\mu\right]_{+}}.$}

 To this end, for any $\varepsilon>0$, pick  $k=\theta_0-\varepsilon$ in  \dref{5677888cz2sssssss.5xx1}, then, $$-\mu+\frac{({k}-1)}{{k}}\max_{s\geq1}\lambda_0^{\frac{1}{{{s}}+1}}(\chi+\xi\|w_0\|_{L^\infty(\Omega)})<0~~~(\mbox{by}~~0<\mu<\mu<\max_{s\geq1}
 \lambda_0^{\frac{1}{{{s}}+1}}(\chi+\xi\|w_0\|_{L^\infty(\Omega)}),$$
 so that, \dref{5677888cz2sssssss.5xx1} implies that for some positive constant $C_1$ such that
 \begin{equation}
\begin{array}{rl}
&\disp{\frac{1}{{k}}\|u(\cdot,t) \|^{{{k}}}_{L^{{k}}(\Omega)}}
\\
\leq&\disp{\frac{1}{2}[\frac{({k}-1)}{{k}}\max_{s\geq1}\lambda_0^{\frac{1}{{{s}}+1}}(\chi+\xi\|w_0\|_{L^\infty(\Omega)})- \mu]\int_{0}^t
e^{-( {k}+1)(t-s)}\int_\Omega u^{{{k}+1}}dxds +C_1}\\
\end{array}
\label{czssddssddffggf2sssss.5ksdddk1214114114rrggkkll}
\end{equation}
by using the Young inequality.
This combined with  the arbitrariness of $\varepsilon$ and the H\"{o}lder  inequality yields to for any $k_0\in (\max\{1,\theta_0-\frac{N}{2}\},\theta_0)$,
\begin{equation}
\|u(\cdot, t)\|_{L^{k_0}(\Omega)}\leq C_2~~\mbox{for all}~~ t\in(0,T_{max})
\label{zjscz2.5297x96ssdddd30ssss111}
\end{equation}
and some positive constant $C_2$.

{\bf Step 2. The boundedness of $\|u(\cdot, t)\|_{L^{k}(\Omega)}$
for all $t\in (0, T_{max})$ and $k\in (1,\theta_0]$, where $\theta_0=\frac{\max_{s\geq1}\lambda_0^{\frac{1}{{{s}}+1}}(\chi+\xi\|w_0\|_{L^\infty(\Omega)})}{\left[\max_{s\geq1}
 \lambda_0^{\frac{1}{{{s}}+1}}(\chi+\xi\|w_0\|_{L^\infty(\Omega)})-\mu\right]_{+}}$.}

To achieve this, we pick $k=\theta_0$ in  \dref{5677888cz2sssssss.5xx1}, then,
 $-\mu+\frac{({k}-1)}{{k}}\max_{s\geq1}\lambda_0^{\frac{1}{{{s}}+1}}(\chi+\xi\|w_0\|_{L^\infty(\Omega)})=0$, so that, by \dref{5677888cz2sssssss.5xx1}, we have
 \begin{equation}
\begin{array}{rl}
&\disp{\frac{1}{{k}}\|u(\cdot,t) \|^{{{k}}}_{L^{{k}}(\Omega)}+C_ D(k-1)\int_{0}^t
e^{-( { {k}+1})(t-s)}\int_{\Omega}u^{m+k-3}|\nabla u|^2dxds}
\\
\leq&\disp{ \rho_0\int_{0}^t
e^{-( { {k}+1})(t-s)}\int_\Omega u^{{{k}}} dxds+\rho_1~~\mbox{for all}~~ t\in(0,T_{max}).}\\
\end{array}
\label{czssddssddffggf2.5kk1214114114rrsssddggkkll}
\end{equation}
In the following, we shall apply the Gagliardo-Nirenberg interpolation inequality to
control the second integral on the right-hand side of \dref{czssddssddffggf2.5kk1214114114rrsssddggkkll}.
To this end, in view of $m\geq2-\frac{2}{N}\theta_0$ and $k_0>\max\{1,\theta_0-\frac{N}{2}\}$ implies that
$$m+k-1+\frac{2}{N}\times k_0>k,$$
therefore, in view of \dref{zjscz2.5297x96ssdddd30ssss111},
 we deduce from the Gagliardo--Nirenberg inequality 
 that
  there exist positive constants  $C_{3}$ and $C_4$ such that
\begin{equation}
\begin{array}{rl}
&\disp \rho_0\int_{\Omega} u^{k}\\
=&\disp{\|  u^{\frac{m+k-1}{2}}\|^{\frac{2{k}}{m+k-1}}_{L^{\frac{2{k}}{m+k-1}}(\Omega)}}\\
\leq&\disp{C_{3}(\|\nabla    u^{\frac{m+k-1}{2}}\|_{L^2(\Omega)}^{\frac{m+k-1}{k}\frac{N(k-k_0)}{N(m+k-1)+(2-N)k_0}}\|   u^{\frac{m+k-1}{2}}\|_{L^\frac{2k_0}{m+k-1 }(\Omega)}^{1-\frac{m+k-1}{k}\frac{N(k-k_0)}{N(m+k-2)+(2-N)k_0}}}\\
&\disp{+\|   u^{\frac{m+k-1}{2}}\|_{L^\frac{2k_0}{m+k-1 }(\Omega)})^{\frac{2{k}}{m+k-1}}}\\
\leq&\disp{C_{4}(\|\nabla    u^{\frac{m+k-1}{2}}\|_{L^{2}(\Omega)}^{\frac{2N(k-k_0)}{N(m+k-1)+(2-N)k_0}}+1),}\\
\end{array}
\label{ssddd123cz2.57151hhdfkkhjdsssdffgukildrftjj}
\end{equation}
which, together with the fact
$${\frac{2N(k-k_0)}{N(m+k-1)+(2-N)k_0}}<2 ~(\mbox{by}~~ m+k-1+\frac{2}{N}\times k_0>k),$$
immediately gives that
$$
\begin{array}{rl}
\disp \rho_0\int_{\Omega} u^{k}\leq\disp{\frac{2C_ D(k-1)}{(m+k-1)^2}\|\nabla    u^{\frac{m+k-1}{2}}\|_{L^{2}(\Omega)}^{2}+C_5}\\
\end{array}
$$
and some positive constant $C_5$.
Substituting the above inequality into \dref{czssddssddffggf2.5kk1214114114rrsssddggkkll}, we  can get
\dref{qqqqzjscz2.5297x9630111}.

%

{\bf Case $\mu\geq\max_{s\geq1}
 \lambda_0^{\frac{1}{{{s}}+1}}(\chi+\xi\|w_0\|_{L^\infty(\Omega)})$:} 
For any $\varepsilon>0,$
we choose $$k=\frac{\max_{s\geq1}\lambda_0^{\frac{1}{{{s}}+1}}(\chi+\xi\|w_0\|_{L^\infty(\Omega)})}
{\left[\max_{s\geq1}\lambda_0^{\frac{1}{{{s}}+1}}(\chi+\xi\|w_0\|_{L^\infty(\Omega)})-\mu\right]_{+}}-\varepsilon.$$
Then $$\frac{({k}-1)}{{k}}\max_{s\geq1}\lambda_0^{\frac{1}{{{s}}+1}}(\chi+\xi\|w_0\|_{L^\infty(\Omega)})<\mu,$$
so that,  \dref{5677888cz2sssssss.5xx1} yields to for some positive constant $C_6$ such that
$$
\begin{array}{rl}
&\disp{\frac{1}{{k}}\|u(\cdot,t) \|^{{{k}}}_{L^{{k}}(\Omega)}+C_ D(k-1)\int_{0}^t
e^{-( { {k}+1})(t-s)}\int_{\Omega}u^{m+k-3}|\nabla u|^2dxds }
\\
&+\disp{\frac{1}{2}[\mu-\frac{({k}-1)}{{k}}\max_{s\geq1}
 \lambda_0^{\frac{1}{{{s}}+1}}(\chi+\xi\|w_0\|_{L^\infty(\Omega)})]\int_{0}^t
e^{-( {k}+1)(t-s)}\int_\Omega u^{{{k}+1}}dxds }\\
\leq&\disp{C_6~~~\mbox{for all}~~t\in (0, T_{max})}\\
\end{array}
$$
by using Young inequality.
The proof Lemma \ref{qqqqlemma45630} is completed. 
\end{proof}
.
\begin{lemma}\label{lessdddmma45630}
Let $(u,v,w)$ be a solution to \dref{1.1} on $(0,T_{max})$. If  $0<\mu<\kappa_0$ and $m\geq 2-\frac{2}{N}\lambda$,
then there exist positive constants $\alpha_0>\lambda$ and $C$  which satisfy
\begin{equation}
\|u(\cdot, t)\|_{L^{\alpha_0}(\Omega)}\leq C ~~\mbox{for all}~~ t\in(0,T_{max}),
\label{zjscz2.529dfgggg7x9630111}
\end{equation}
 where
\begin{equation}
 \lambda=\frac{\kappa_0}{(\kappa_0-\mu)_+}
\label{3344ddff4dffgg4zjscz2.5297x96302222114}
\end{equation}
and
 $\kappa_0$ is the same as
\dref{113344ddff4dffgg4zjscz2ddd.5297x96302222114}.
\end{lemma}
\begin{proof}
Firstly, by \dref{3344ddff4dffgg4zjscz2.5297x96302222114}
and \dref{113344ddff4dffgg4zjscz2ddd.5297x96302222114}, 
in view of Lemma \ref{lemma45630} and Lemma \ref{qqqqlemma45630} , we deduce that
\begin{equation}
\|u(\cdot, t)\|_{L^{\lambda}(\Omega)}\leq C_1 ~~\mbox{for all}~~ t\in(0,T_{max})
\label{zjscz2.529dfgggg7x9630dddd111}
\end{equation}
and some positive constant $C_1,$  where $\lambda$ is given by \dref{3344ddff4dffgg4zjscz2.5297x96302222114}.
On the other hand,   by Lemma  \ref{3455667lemma45630}, we obtain that
$$
\begin{array}{rl}
&\disp{\frac{1}{{k}}\|u(\cdot,t) \|^{{{k}}}_{L^{{k}}(\Omega)}+C_ D(k-1)\int_{0}^t
e^{-( { {k}+1})(t-s)}\int_{\Omega}u^{m+k-3}|\nabla u|^2dxds}
\\
\leq&\disp{[\frac{({k}-1)}{{k}}\kappa_0- \mu]\int_{0}^t
e^{-( {k}+1)(t-s)}\int_\Omega u^{{{k}+1}}dxds }\\
&+\disp{ \rho_0\int_{0}^t
e^{-( { {k}+1})(t-s)}\int_\Omega u^{{{k}}} dxds+\rho_1~~~\mbox{for all}~~t\in (0, T_{max}),}\\
\end{array}
$$
where $\kappa_0,\rho_0$ and $\rho_1$ are the same as Lemma  \ref{3455667lemma45630}.
This combined with the Young inequality implies that for any $\delta_1>0,$
\begin{equation}
\begin{array}{rl}
&\disp{\frac{1}{{k}}\|u(\cdot,t) \|^{{{k}}}_{L^{{k}}(\Omega)}+C_ D(k-1)\int_{0}^t
e^{-( { {k}+1})(t-s)}\int_{\Omega}u^{m+k-3}|\nabla u|^2dxds}
\\
\leq&\disp{[\frac{({k}-1)}{{k}}\kappa_0+\delta_1- \mu]\int_{0}^t
e^{-( {k}+1)(t-s)}\int_\Omega u^{{{k}+1}}dxds +C_1~~~\mbox{for all}~~t\in (0, T_{max})}\\
\end{array}
\label{czssddssddssfffgffggf2.3456675kk1214114114rrggkkll}
\end{equation}
and  some positive constant $C_1$.
Next,  in view of \dref{zjscz2.529dfgggg7x9630dddd111},
 we conclude from the Gagliardo--Nirenberg inequality 
 that
  there exist positive constants  $C_{2}=C_{2}(k)$ and $C_3=C_{3}(k)$ such that
\begin{equation}
\begin{array}{rl}
&\disp \int_{\Omega} u^{k+1}\\
=&\disp{\|  u^{\frac{m+k-1}{2}}\|^{\frac{2{(k+1)}}{m+k-1}}_{L^{\frac{2{(k+1)}}{m+k-1}}(\Omega)}}\\
\leq&\disp{C_{2}(\|\nabla    u^{\frac{m+k-1}{2}}\|_{L^2(\Omega)}^{\frac{N(m+k-1)}{k+1}\frac{k+1-\lambda}{\lambda(2-N)+N(m+k-1)}}\|   u^{\frac{m+k-1}{2}}\|_{L^\frac{2\lambda}{m+k-1 }(\Omega)}^{1-\frac{N(m+k-1)}{k+1}\frac{k+1-\lambda}{\lambda(2-N)+N(m+k-1)}}}\\
&\disp{+\|   u^{\frac{m+k-1}{2}}\|_{L^\frac{2\lambda}{m+k-1 }(\Omega)})^{\frac{2{(k+1)}}{m+k-1}}}\\
\leq&\disp{C_{3}(\|\nabla    u^{\frac{m+k-1}{2}}\|_{L^{2}(\Omega)}^{\frac{2N(k+1-\lambda)}{\lambda(2-N)+N(m+k-1)}}+1),}\\
\end{array}
\label{123cz2.57151hhdfkkhjsssdsssdffgukildrftjj}
\end{equation}
where  $\lambda$ is the same as  \dref{3344ddff4dffgg4zjscz2.5297x96302222114}.
Since $m\geq 2-\frac{2}{N}\lambda$, one can easily see that
$${\frac{2N(k+1-\lambda)}{\lambda(2-N)+N(m+k-1)}}\leq2,$$
so that, \dref{123cz2.57151hhdfkkhjsssdsssdffgukildrftjj} yields to
\begin{equation}
\begin{array}{rl}
 \disp \int_{\Omega} u^{k+1}\leq&\disp{C_{4}(\|\nabla    u^{\frac{m+k-1}{2}}\|_{L^{2}(\Omega)}^{2}+1)}\\
\end{array}
\label{123cz2.57151hhdfkkhjsssdsssdffddddgukildrftjj}
\end{equation}
for some positive constant $C_4(k)>0.$
Substituting \dref{123cz2.57151hhdfkkhjsssdsssdffddddgukildrftjj} into \dref{czssddssddssfffgffggf2.3456675kk1214114114rrggkkll}, we obatin that
\begin{equation}
\begin{array}{rl}
&\disp{\frac{1}{{k}}\|u(\cdot,t) \|^{{{k}}}_{L^{{k}}(\Omega)}}\\
&\disp{+[\frac{4C_ D(k-1)}{(m+k-1)^2}-(-\mu+\frac{k-1}{k }\kappa_0)C_4+\delta_1C_4]\int_{0}^t
e^{-( { {k}+1})(t-s)}\|\nabla (u+1)^{\frac{m+k-1}{2}}\|^2_{L^2(\Omega)}ds}
\\
\leq&\disp{C_5~~~\mbox{for all}~~t\in (0, T_{max})}\\
\end{array}
\label{ssdddcz2ssssssssesddffrsssssrrs.5xx1}
\end{equation}
and some positive constant $C_5$.
Let $k>\lambda$. Then by some basic calculation, we derive that
%
\begin{equation}\lim_{k\nearrow\lambda}[\frac{4C_ D(k-1)}{(m+k-1)^2}-(-\mu+\frac{k-1}{k }\kappa_0)C_4]=\frac{4C_ D(\lambda-1)}{(m+\lambda-1)^2}>0
\label{cz2.5sssssxsdddfffsdddx1}
\end{equation}
and
\begin{equation}\frac{4C_ D(k-1)}{(m+k-1)^2}-(-\mu+\frac{k-1}{k }\kappa_0)C_4>0~~~\mbox{for all}~~k>\lambda,
\label{cz2sfggg.5sssssxsdddsdddx1}
\end{equation}
where we have used the fact that $\mu<\kappa_0$.
Collecting \dref{123cz2.57151hhdfkkhjsssdsssdffddddgukildrftjj}--\dref{cz2sfggg.5sssssxsdddsdddx1},  we may choose $k>\lambda$ which is close to $\lambda$ such that
\begin{equation}\lim_{k\nearrow\lambda}[\frac{4C_ D(k-1)}{(m+k-1)^2}-(\mu+\frac{k-1}{k }\kappa_0)C_4]=\frac{2C_ D(\lambda-1)}{(m+\lambda-1)^2}.
\label{cz2.5sssssxsddssdddddfffsdddx1}
\end{equation}
Next, substitute \dref{cz2.5sssssxsddssdddddfffsdddx1} into \dref{ssdddcz2ssssssssesddffrsssssrrs.5xx1} and choose $\delta_1$ suitablely small (e.g.
$\delta_1<\frac{2C_ D(\lambda-1)}{C_4(m+\lambda-1)^2}$), then we have
\begin{equation}
\begin{array}{rl}
\|u(\cdot,t) \|_{L^{{k}}(\Omega)}\leq&\disp{C_6~~~\mbox{for all}~~t\in (0, T_{max})}\\
\end{array}
\label{ssdddcz2ssssssssssssesddffrsssssrrs.5xx1}
\end{equation}
and the proof of Lemma \ref{lessdddmma45630} is thus completed.
%
\end{proof}
Along with the basic estimate from Lemmas \ref{lemma45630}--\ref{lessdddmma45630}, this immediately implies the following Lemma:
\begin{lemma}\label{lemmasdffssdd45566645630223}
Assume that  $\mu>0$.
If
  \begin{equation}\label{gddffffnjjmmx1.731426677gg}
m\geq 2-\frac{2}{N}\lambda~~~\mbox{with}~~\mu<\kappa_0,
\begin{array}{ll}\\
 \end{array}
\end{equation}
then for $k>\max\{N+1,N(m+1)\}$, there exists a positive constant $C=C(k,|\Omega|,\mu,\lambda_0,\xi,\chi,m,C_{D} )$  
such that 
 the solution of \dref{1.1} from Lemma \ref{lemma70} satisfies
\begin{equation}
\int_{\Omega}u^{k}(x,t)dx\leq C ~~~\mbox{for all}~~ t\in(0,T_{max}),
\label{33444dffgg4zjscz2.5297x96302222114}
\end{equation}
where
$\lambda$ and $\kappa_0$ are given by \dref{3344ddff4dffgg4zjscz2.5297x96302222114}
and \dref{113344ddff4dffgg4zjscz2ddd.5297x96302222114}, respectively.

\end{lemma}
\begin{proof}
Firstly, due to Lemma \ref{lessdddmma45630}, we derive that there exists a positive constant $C_1$ such that
\begin{equation}
\|u(\cdot, t)\|_{L^{\alpha_0}(\Omega)}\leq C_1 ~~\mbox{for all}~~ t\in(0,T_{max}),
\label{zjscz2.529dfgggg7x9ssdccvvvdf630111}
\end{equation}
where $\alpha_0$ is the same as Lemma \ref{lessdddmma45630}.
 Next,
 Lemma  \ref{3455667lemma45630} implies that
$$
\begin{array}{rl}
&\disp{\frac{1}{{k}}\|u(\cdot,t) \|^{{{k}}}_{L^{{k}}(\Omega)}+C_ D(k-1)\int_{0}^t
e^{-( { {k}+1})(t-s)}\int_{\Omega}u^{m+k-3}|\nabla u|^2dxds}
\\
\leq&\disp{[\frac{({k}-1)}{{k}}\kappa_0- \mu]\int_{0}^t
e^{-( {k}+1)(t-s)}\int_\Omega u^{{{k}+1}}dxds }\\
&+\disp{ \rho_0\int_{0}^t
e^{-( { {k}+1})(t-s)}\int_\Omega u^{{{k}}} dxds+\rho_1~~~\mbox{for all}~~t\in (0, T_{max}),}\\
\end{array}
$$
where $\kappa_0,\rho_0$ and $\rho_1$ are the same as Lemma  \ref{3455667lemma45630}.
Therefore, we derive from the Young inequality that there exists a positive constant $C_1$ such that
\begin{equation}
\begin{array}{rl}
&\disp{\frac{1}{{k}}\|u(\cdot,t) \|^{{{k}}}_{L^{{k}}(\Omega)}+C_ D(k-1)\int_{0}^t
e^{-( { {k}+1})(t-s)}\int_{\Omega}u^{m+k-3}|\nabla u|^2dxds}
\\
\leq&\disp{\kappa_0\int_{0}^t
e^{-( {k}+1)(t-s)}\int_\Omega u^{{{k}+1}}dxds+C_1~~~\mbox{for all}~~t\in (0, T_{max}).}\\
\end{array}
\label{1113333cz2.511sddff4114}
\end{equation}
For any $k>\max\{N+1,N(m+1),\alpha_0-1,1,1-m+\frac{N-2}{N}\alpha_0\}$, $m\geq2-\frac{2}{N}\lambda $  together with $\alpha_0>\lambda$
yields to
 $$k+1\leq m+k-1+\frac{2}{N}\lambda<m+k-1+\frac{2}{N}\alpha_0,$$
so that, in particular, according to
by the Gagliardo--Nirenberg inequality and \dref{zjscz2.529dfgggg7x9ssdccvvvdf630111}, 
one can get there exist positive constants  $C_2$ and
$C_3$ 
such that
\begin{equation}
\begin{array}{rl}
&\disp\kappa_0\int_{\Omega} u^{k+1}\\
=&\disp{\kappa_0\|  u^{\frac{m+k-1}{2}}\|^{\frac{2(k+1)}{m+k-1}}_{L^{\frac{2(k+1)}{m+k-1}}(\Omega)}}\\
\leq&\disp{C_{2}(\|\nabla    u^{\frac{m+k-1}{2}}\|_{L^2(\Omega)}^{\frac{m+k-1}{k+1}}\|   u^{\frac{m+k-1}{2}}\|_{L^\frac{2\alpha_0}{m+k-1 }(\Omega)}^{1-\frac{m+k-1}{k+1}}+\|   u^{\frac{m+k-1}{2}}\|_{L^\frac{2\alpha_0}{m+k-1 }(\Omega)})^{\frac{2(k+1)}{m+k-1}}}\\
\leq&\disp{C_{3}(\|\nabla    u^{\frac{m+k-1}{2}}\|_{L^{2}(\Omega)}^{2\frac{N(k+1)-N\alpha_0}{(2-N)\alpha_0+N(m+k-1)}}+1).}\\
\end{array}
\label{123cz2.57151hhdhhjjjdfffkkhhhjddffffgukildrftjj}
\end{equation}
In view of $m\geq2-\frac{2}{N}\lambda $ and $\alpha_0>\lambda$, by some basic
calculation,
we derive that $$\frac{N(k+1)-N\alpha_0}{(2-N)\alpha_0+N(m+k-1)}<1,$$
so that, which returns, using again the Young inequality,
%
for any $\delta_1>0$,
\begin{equation}
\begin{array}{rl}
&\disp\kappa_0\int_{\Omega} u^{k+1}\leq
\disp{\delta_1\|\nabla    u^{\frac{m+k-1}{2}}\|_{L^{2}(\Omega)}^{2}+C_{4}.}\\
\end{array}
\label{123cz2.571sddff51hhdhhjdfffkkhjdfffgukildrftjj}
\end{equation}
Combining the above three estimates and choosing $\delta_1$ appropriately small, we
arrive at for some positive constant $C_{5}$ such that   
  \begin{equation}
\int_{\Omega}u^{k}(x,t)dx\leq C_{5} ~~~\mbox{for all}~~ t\in(0,T_{max}),
\label{334444zjscz2.5297dfggggx96302222114}
\end{equation}
from which we readily infer \dref{33444dffgg4zjscz2.5297x96302222114}.
The proof of Lemma \ref{lemmasdffssdd45566645630223} is completed.
\end{proof}

\begin{lemma}\label{lemdddmasdffssdd45566645630223}
Assume that  $\mu>0$.
If
  \begin{equation}\label{gddffffnjjmmfffffhhjx1.731426677gg}
m>2-\frac{2}{N}\lambda~~~\mbox{and}~~\mu\geq\kappa_0,
\begin{array}{ll}\\
 \end{array}
\end{equation}
then 
for all $p>\max\{N+1,N(m+1)\}$, there exists a positive constant $C=C(p,|\Omega|,\mu,\lambda_0,\xi,\chi,m,C_{D} )$  
such that 
 the solution of \dref{1.1} from Lemma \ref{lemma70} satisfies
\begin{equation}
\int_{\Omega}u^{p}(x,t)dx\leq C ~~~\mbox{for all}~~ t\in(0,T_{max}),
\label{33444dfddrryuufgg4zjscz2.5297x96302222114}
\end{equation}
where $\lambda$ and $\kappa_0$ are given by \dref{3344ddff4dffgg4zjscz2.5297x96302222114}
and \dref{113344ddff4dffgg4zjscz2ddd.5297x96302222114}, respectively.
\end{lemma}
\begin{proof}
 Firstly, in view of Lemma \ref{qqqqlemma45630}, there exists a positive constant $C_1$ such that
\begin{equation}
\int_{\Omega}u^{l_0}(x,t)dx\leq C_{1} ~~~\mbox{for all}~~ t\in(0,T_{max}),
\label{llll33444ddrr4zjscz2.5297dfggggx963022221sdddd14}
\end{equation}
where $l_0=\lambda -\varepsilon$ with $\varepsilon=\frac{1}{3}\frac{N}{2}(m-2+\frac{2}{N}\lambda) $.
Next, by \dref{1113333cz2.511sddff4114}, we also have
\begin{equation}
\begin{array}{rl}
&\disp{\frac{1}{{k}}\|u(\cdot,t) \|^{{{k}}}_{L^{{k}}(\Omega)}+C_ D(k-1)\int_{0}^t
e^{-( { {k}+1})(t-s)}\int_{\Omega}u^{m+k-3}|\nabla u|^2dxds}
\\
\leq&\disp{\kappa_0\int_{0}^t
e^{-( {k}+1)(t-s)}\int_\Omega u^{{{k}+1}}dxds+C_1~~~\mbox{for all}~~t\in (0, T_{max}),}\\
\end{array}
\label{1113333csdfffz2.511sddff4114}
\end{equation}
where $C_1$ is the same as \dref{1113333cz2.511sddff4114}.
On the other hand, since $m>2-\frac{2}{N}\lambda ,$ yields to $p+1<m+p-1+\frac{2}{N}l_0,$
so that, in particular, according to
by the Gagliardo--Nirenberg inequality and \dref{llll33444ddrr4zjscz2.5297dfggggx963022221sdddd14}, 
one can get there exist positive constants  $C_2$ and
$C_3$ 
such that
\begin{equation}
\begin{array}{rl}
&\disp\kappa_0\int_{\Omega} u^{k+1}\\
=&\disp{\kappa_0\|  u^{\frac{m+k-1}{2}}\|^{\frac{2(k+1)}{m+k-1}}_{L^{\frac{2(k+1)}{m+k-1}}(\Omega)}}\\
\leq&\disp{C_2(\|\nabla    u^{\frac{m+k-1}{2}}\|_{L^2(\Omega)}^{\frac{m+k-1}{k+1}}\|   u^{\frac{m+k-1}{2}}\|_{L^\frac{2l_0}{m+k-1 }(\Omega)}^{1-\frac{m+k-1}{k+1}}+\|   u^{\frac{m+k-1}{2}}\|_{L^\frac{2l_0}{m+k-1 }(\Omega)})^{\frac{2(k+1)}{m+k-1}}}\\
\leq&\disp{C_3(\|\nabla    u^{\frac{m+k-1}{2}}\|_{L^{2}(\Omega)}^{2\frac{N(k+1)-Nl_0}{(2-N)l_0+N(m+k-1)}}+1)}\\
\end{array}
\label{123cz2.57151hhdhhjjjdfffkkhhhjddffffgukildrftjj}
\end{equation}
This  together with $\frac{N(k+1)-Nl_0}{(2-N)l_0+N(m+k-1)}<1$ (by $m>2-\frac{2}{N}\lambda $) and the Young ineuqality implies that
%
for any $\delta_1>0$,
\begin{equation}
\begin{array}{rl}
&\disp\kappa_0\int_{\Omega} u^{k+1}\leq
\disp{\delta_1\|\nabla    u^{\frac{m+k-1}{2}}\|_{L^{2}(\Omega)}^{2}+C_{4},}\\
\end{array}
\label{123cz2.571sddff51hhdhhjdfffkkhjdfffgukildrftjj}
\end{equation}
which combined with \dref{1113333csdfffz2.511sddff4114} implies that
\begin{equation}
\int_{\Omega}u^{k}(x,t)dx\leq C_{5} ~~~\mbox{for all}~~ t\in(0,T_{max})
\label{33444ddrr4zjscz2.529dddddg7dfggggx96302222114}
\end{equation}
by picking  $\delta_1$ appropriately small in \dref{123cz2.571sddff51hhdhhjdfffkkhjdfffgukildrftjj}.
 Finally,
using  the H\"{o}lder inequality, we can get \dref{33444dfddrryuufgg4zjscz2.5297x96302222114}.
The proof of Lemma \ref{lemdddmasdffssdd45566645630223} is completed.
\end{proof}

\begin{lemma}\label{lemma4ddffffghhhggg5630223116} 
Let \begin{equation}\label{eqllooox45xx121ddffrttghhhiigghhjjoo12}
C_{D}>\frac{C_{GN}(1+\|u_0\|_{L^1(\Omega)})^3}{4}(2-\frac{2}{N})^2\kappa_0
 \end{equation}
 and $$h(p) :=\frac{4C_{D}}{C_{GN}(1+\|u_0\|_{L^1(\Omega)})^3}-\frac{(1-\frac{2}{N}+p)^2}{{p}}\kappa_0, $$
 where
 $\kappa_0$ is the same as \dref{113344ddff4dffgg4zjscz2ddd.5297x96302222114},
 $p\geq1,C_{D}$ and $C_{GN}$ are positive constants.
Then there exists a positive constant $\tilde{p}_0>1$ such that
\begin{equation}\label{eqx45xx121ddffrttghhhiigghhjjoo12}
h(p)>0~~~\mbox{for all}~~p\in(1,\tilde{p}_0],
\end{equation}

\end{lemma}
\begin{proof}
The idea comes from \cite{Zhengssddff0}. Indeed,
due to \dref{eqllooox45xx121ddffrttghhhiigghhjjoo12}, it is not difficult to verify that $h(1)\geq\frac{4C_{D}}{C_{GN}(1+\|u_0\|_{L^1(\Omega)})^3}-(2-\frac{2}{N})^2\kappa_0>0$.
Next, by basic calculation, we derive that for any $p\geq1,$
$h'(p)=\frac{(1-\frac{2}{N}+p)(p+\frac{2}{N}-1)}{p^2}\kappa_0<0.$
Therefore,  from the monotonicity of  $h$, there exists a positive constant $\tilde{p}_0>1$ such that   \dref{eqx45xx121ddffrttghhhiigghhjjoo12} holds.
\end{proof}

\begin{lemma}\label{lemma45566645630223}
Assume that  $\mu=0$.
If
  \begin{equation}\label{gddffffnjjmmffffffjjjffffhhjx1.731426677gg}
m> 2-\frac{2}{N}
\begin{array}{ll}\\
 \end{array}
\end{equation}
or
 \begin{equation}\label{ddffffgddffffnjjmmx1.731426677gg}
m= 2-\frac{2}{N}~~\mbox{and}~~~C_{D} >\frac{C_{GN}(1+\|u_0\|_{L^1(\Omega)})^3}{4}(2-\frac{2}{N})^2\kappa_0,
\begin{array}{ll}\\
 \end{array}
\end{equation}
then there exists a positive constant $p_0>1$   
such that 
 the solution of \dref{1.1} from Lemma \ref{lemma70} satisfies
\begin{equation}
\int_{\Omega}u^{p_0}(x,t)dx\leq C ~~~\mbox{for all}~~ t\in(0,T_{max}),
\label{33sdfffff4444zjscz2.5297x96302222114}
\end{equation}
where $\kappa_0$ is the same as \dref{113344ddff4dffgg4zjscz2ddd.5297x96302222114}.
\end{lemma}
\begin{proof}
Without loss of generality, we may assume that
$$m= 2-\frac{2}{N}~~\mbox{and}~~~C_{D} >\frac{C_{GN}(1+\|u_0\|_{L^1(\Omega)})^3}{4}(2-\frac{2}{N})^2\kappa_0,$$
since, $m>2-\frac{2}{N}$ can  be proved similarly and easily.
Assume that $\tilde{p}_0$ is the same as lemma \ref{lemma4ddffffghhhggg5630223116} and let $1<k\leq\min\{2,\tilde{p}_0\}$.
Then due to  \dref{czssddssddffggf2.3456675kk1214114114rrggkkll} and $\mu=0$, we also derive that for the above $k,$
\begin{equation}
\begin{array}{rl}
&\disp{\frac{1}{{k}}\|u(\cdot,t) \|^{{{k}}}_{L^{{k}}(\Omega)}+C_ D(k-1)\int_{0}^t
e^{-( { {k}+1})(t-s)}\int_{\Omega}u^{m+k-3}|\nabla u|^2dxds}
\\
\leq&\disp{\frac{({k}-1)}{{k}}\kappa_0\int_{0}^t
e^{-( {k}+1)(t-s)}\int_\Omega u^{{{k}+1}}dxds }\\
&+\disp{ \rho_0\int_{0}^t
e^{-( { {k}+1})(t-s)}\int_\Omega u^{{{k}}} dxds+\rho_1~~~\mbox{for all}~~t\in (0, T_{max}),}\\
\end{array}
\label{czssddssddffggf2.3456675kk1ssddf214114114rrggkkll}
\end{equation}
where $\kappa_0,\rho_0$ and $\rho_1$ are the same as Lemma  \ref{3455667lemma45630}.
Here, in order to estimate the rightmost term appropriately, we
employ the Gagliardo-Nirenberg inequality to obtain $C_{GN}>0$ such that
\begin{equation}
\begin{array}{rl}
&\disp\int_{\Omega} u^{k+1}\\
=&\disp{\|  u^{\frac{1-\frac{2}{N}+k}{2}}\|^{\frac{2(k+1)}{1-\frac{2}{N}+k}}_{L^{\frac{2(k+1)}{1-\frac{2}{N}+k}}(\Omega)}}\\
\leq&\disp{C_{GN}(\|\nabla    u^{\frac{1-\frac{2}{N}+k}{2}}\|_{L^2(\Omega)}^{\frac{1-\frac{2}{N}+k}{k+1}}\|   u^{\frac{1-\frac{2}{N}+k}{2}}\|_{L^\frac{2}{1-\frac{2}{N}+k }(\Omega)}^{1-\frac{1-\frac{2}{N}+k}{k+1}}+\|   u^{\frac{1-\frac{2}{N}+k}{2}}\|_{L^\frac{2}{1-\frac{2}{N}+k }(\Omega)})^{\frac{2(k+1)}{1-\frac{2}{N}+k}}}\\
\leq&\disp{C_{GN}(1+\|u_0\|_{L^1(\Omega)})^3(\|\nabla    u^{\frac{1-\frac{2}{N}+k}{2}}\|_{L^2(\Omega)}^{2}+1)}\\
\end{array}
\label{123cz2.57151hhddfffkkhhhjddffffgukildrftjj}
\end{equation}
by using \dref{333ssddaqwswddaassffssff3.ddfvbb10deerfgghhjuuloollgghhhyhh},
where in the last inequality we have  used $k\leq2$ and
 $C_{GN}$ is the same as Lemma \ref{lemma41}. 
In combination with \dref{czssddssddffggf2.3456675kk1ssddf214114114rrggkkll} and \dref{123cz2.57151hhddfffkkhhhjddffffgukildrftjj}, this shows that
\begin{equation}
\begin{array}{rl}
&\disp{\frac{1}{{k}}\|u(\cdot,t) \|^{{{k}}}_{L^{{k}}(\Omega)}}\\
&\disp{+(\frac{4C_{D}(k-1)}{(1-\frac{2}{N}+k)^2}\frac{1}{C_{GN}(1+\|u_0\|_{L^1(\Omega)})^3}-\frac{({k}-1)}{{k}}\kappa_0)\int_{0}^t
e^{-( {k}+1)(t-s)}\int_\Omega u^{{{k}+1}}dxds}
\\
\leq&\disp{\rho_0\int_{0}^t
e^{-( { {k}+1})(t-s)}\int_\Omega u^{{{k}}} dxds+C_1~~~\mbox{for all}~~t\in (0, T_{max}).}\\
\end{array}
\label{czssddssddffggf2.ssddd3456675kk1ssddf214114114rrggkkll}
\end{equation}
Therefore,
 by \dref{ddffffgddffffnjjmmx1.731426677gg} and $1<k\leq\min\{2,\tilde{p}_0\}$,
 $$\frac{4C_{D}(k-1)}{(1-\frac{2}{N}+k)^2}\frac{1}{C_{GN}(1+\|u_0\|_{L^1(\Omega)})^3}-\frac{({k}-1)}{{k}}\kappa_0>0.$$
 %
      Hence, using  Lemma  \ref{lemma4ddffffghhhggg5630223116}, we derive from the Young inequality and \dref{czssddssddffggf2.ssddd3456675kk1ssddf214114114rrggkkll} that there exist positive constants $p_0>1$ and $C_2$ such that 
  \begin{equation}
\int_{\Omega}u^{p_0}(x,t)dx\leq C_{2} ~~~\mbox{for all}~~ t\in(0,T_{max}).
\label{33444ddrr4zjsczccvv2.5297dfggggx96302222114}
\end{equation}
The proof of Lemma \ref{lemma45566645630223} is completed.
\end{proof}

Our next goal is to make sure that  Lemma \ref{lemma45566645630223}  is sufficient to enforce boundedness of $\|u(\cdot,t)\|_{L^p(\Omega)}$ for all $t\in(0, T_{max})$ and $p>1$, which
plays  a key step in the derivation of our main results.
\begin{lemma}\label{lemmaddffddfffrsedrffffffgg}
Suppose that the conditions of Lemma \ref{lemma45566645630223} hold.
Then for any $p>1,$ there exists a positive constant $C:=C(p,|\Omega|,C_{D},C_{GN},\lambda_0,m,\chi)$ such that 
\begin{equation}
 \begin{array}{rl}
 \|u(\cdot,t)\|_{L^p(\Omega)}\leq C~~~\mbox{for all}~~t\in (0, T_{max}).
\end{array}\label{cz2sfffffedfgg.5g5gghh56789hhjui7ssddd8jj90099}
\end{equation}
\end{lemma}
\begin{proof}
Firstly, let $k>\max\{N+1,N(m+1),p_0-1,1,1-m+\frac{N-2}{N}p_0\}$, where $p_0>1$  is the same as Lemma \ref{lemma45566645630223}.
In view of \dref{czssddssddffggf2.3456675kk1ssddf214114114rrggkkll}, we have
\begin{equation}
\begin{array}{rl}
&\disp{\frac{1}{{k}}\|u(\cdot,t) \|^{{{k}}}_{L^{{k}}(\Omega)}+C_ D(k-1)\int_{0}^t
e^{-( { {k}+1})(t-s)}\int_{\Omega}u^{m+k-3}|\nabla u|^2dxds}
\\
\leq&\disp{\kappa_0\int_{0}^t
e^{-( {k}+1)(t-s)}\int_\Omega u^{{{k}+1}}dxds }\\
&+\disp{ \rho_0\int_{0}^t
e^{-( { {k}+1})(t-s)}\int_\Omega u^{{{k}}} dxds+\rho_1~~~\mbox{for all}~~t\in (0, T_{max}),}\\
\end{array}
\label{11czssddssddffggf2.3456675kk1ssddf214114114rrggkkll}
\end{equation}
where $\kappa_0,\rho_0$ and $\rho_1$ are the same as Lemma  \ref{3455667lemma45630}.
Next, observe that  $m\geq2-\frac{2}{N}$ and $p_0>1$ yields to $k+1<m+k-1+\frac{2}{N}p_0,$
so that, in view of   the Gagliardo--Nirenberg inequality, \dref{33sdfffff4444zjscz2.5297x96302222114} and using the Young inequality,  
one can get there exist positive constants  $C_1,C_2$ and
$C_3$ 
such that for any $\delta_1>0$
\begin{equation}
\begin{array}{rl}
&\disp\kappa_0\int_{\Omega} u^{k+1}\\
=&\disp{\kappa_0\|  u^{\frac{m+k-1}{2}}\|^{\frac{2(k+1)}{m+k-1}}_{L^{\frac{2(k+1)}{m+k-1}}(\Omega)}}\\
\leq&\disp{C_{1}(\|\nabla    u^{\frac{m+k-1}{2}}\|_{L^2(\Omega)}^{\frac{m+k-1}{k+1}}\|   u^{\frac{m+k-1}{2}}\|_{L^\frac{2p_0}{m+k-1 }(\Omega)}^{1-\frac{m+k-1}{k+1}}+\|   u^{\frac{m+k-1}{2}}\|_{L^\frac{2p_0}{m+k-1 }(\Omega)})^{\frac{2(k+1)}{m+k-1}}}\\
\leq&\disp{C_{2}(\|\nabla    u^{\frac{m+k-1}{2}}\|_{L^{2}(\Omega)}^{2\frac{N(k+1)-Np_0}{(2-N)p_0+N(m+k-1)}}+1)}\\
\leq&\disp{\delta_1\|\nabla    u^{\frac{m+k-1}{2}}\|_{L^{2}(\Omega)}^{2}+C_3,}\\
\end{array}
\label{123cz2.57151hhdhhjjjdfffkkhhhjddffffgukildrftjj}
\end{equation}
where we have used $\frac{N(k+1)-Np_0}{(2-N)p_0+N(m+k-1)}<1$ together with  $m\geq2-\frac{2}{N}$ and $p_0>1$.
Inserting \dref{123cz2.57151hhdhhjjjdfffkkhhhjddffffgukildrftjj} into \dref{11czssddssddffggf2.3456675kk1ssddf214114114rrggkkll},  choosing $\delta_1$ appropriately small and using the H\"{o}lder inequality, we can get \dref{cz2sfffffedfgg.5g5gghh56789hhjui7ssddd8jj90099}.
\end{proof}

\section{The proof of main results}

In this section, we are going to prove our main result. To this end, we  will proceed
in two steps. Firstly, applying  the  standard regularity theory of partial differential equation, we  turn the bounds from Lemma \ref{lemmaddffddfffrsedrffffffgg} into a higher order bound for $\nabla v$.
%
%
%
%

\begin{lemma}\label{lemmaddfffrsedrffffffgg}
Suppose that the conditions of Theorem  \ref{theorem3} (or Theorem \ref{dfffftheorem3}) hold.
Let $T\in (0, T_{max})$ and $(u, v,w)$ be the solution of \dref{1.1}.  Then there exists a constant $C > 0$ independent of $T$ such that the
component $v$ of $(u, v,w)$ satisfies
\begin{equation}
 \begin{array}{rl}
 \| v(\cdot,t)\|_{W^{1,\infty}(\Omega)}\leq C~~~\mbox{for all}~~t\in (0, T).
\end{array}\label{cz2sedfgg.5g5gghh56789hhjui7ssddd8jj90099}
\end{equation}
\end{lemma}
\begin{proof}
Due to
$\|u(\cdot, t)\|_{L^p(\Omega)}$ is bounded for any large $p$ (see Lemma \ref{lemmaddffddfffrsedrffffffgg}), we infer from the  standard regularity theory of parabolic equation (or  elliptic equation, $\tau=0$) (see e.g. \cite{Evans551}) that
\dref{cz2sedfgg.5g5gghh56789hhjui7ssddd8jj90099} holds.
\end{proof}
The previous lemmas at hand, we can now pass to the proof of our main result.
Its proof is
based on a Moser-type iteration (see e.g.  \cite{Tao794} and \cite{Keengwwwwssddghjjkk1}).
\begin{lemma}\label{lemmasedrffffffgg}
Under the assumptions of Theorem \ref{theorem3} (or Theorem \ref{dfffftheorem3}), 
one can find  a positive constant
such that for every $T\in(0, T_{max})$
\begin{equation}
 \begin{array}{rl}
 \|u(\cdot,t)\|_{L^\infty(\Omega)}\leq C~~~\mbox{for all}~~t\in (0, T).
\end{array}\label{ssddaqwddfffhhhhkkswddaassffssff3.ddfvbb10deerfgghhjuuloollgghhhyhh}
\end{equation}
\end{lemma}
\begin{proof}
Throughout the proof of Lemma \ref{lemmasedrffffffgg}, we use $C_i$ $(i\in \mathbb{N})$ to denote
the different positive constants independent of $p$ and $k$ ($k
\in \mathbb{N}).$

Case  $m\geq1:$ For any $p>1,$
multiplying both sides of the first equation in \dref{1.1} by $(u+1)^{p-1}$, integrating over $\Omega$, integrating by parts and using the Young inequality and \dref{cz2sedfgg.5g5gghh56789hhjui7ssddd8jj90099} and \dref{x1.731426677gghh}, we derive that
\begin{equation}
\begin{array}{rl}
&\disp{\frac{1}{{p}}\frac{d}{dt}\|u+1\|^{{p}}_{L^{{p}}(\Omega)}+C_{D}({{p}-1})\int_{\Omega}(u+1)^{m+{{p}-3}}|\nabla u|^2}
\\
\leq&\disp{-\chi\int_\Omega \nabla\cdot( u\nabla v)
  (u+1)^{{p}-1}-\xi\int_\Omega \nabla\cdot( u\nabla w)
  (u+1)^{{p}-1} +
\int_\Omega   (u+1)^{{p}-1}(\mu u-\mu u^2) }\\
\leq&\disp{\chi({p}-1)\int_\Omega  u(u+1)^{{p}-2}|\nabla u||\nabla v|
   -\xi({p-1})\int_\Omega \int_{0}^u \tau(\tau+1)^{p-2}d\tau\Delta w}\\
   &\disp{+\int_\Omega   (u+1)^{{p}-1}(\mu u-\mu u^2) }\\
\leq&\disp{\chi({p}-1)\int_\Omega  u(u+1)^{{p}-2}|\nabla u||\nabla v|}\\
&\disp{+  \frac{\xi({p-1})}{p}\int_\Omega (u+1)^{p-1}(\tau \|w_0\|_{L^\infty(\Omega)}\cdot v(x,t)+\kappa)
   +
\int_\Omega   (u+1)^{{p}-1}(\mu u-\mu u^2) }\\
\leq&\disp{\chi({p}-1)C_1\int_\Omega  (u+1)^{{p}-1}|\nabla u|+C_2\int_\Omega  (u+1)^{{p}-1}+
\int_\Omega   (u+1)^{{p}-1}(\mu u-\mu u^2)}\\
\leq&\disp{\frac{({{p}-1})}{4}\int_{\Omega}(u+1)^{m+{{p}-3}}|\nabla u|^2+\chi^2({p}-1)C_1^2\int_\Omega (u+1)^{{p}+1-m}+(C_2+\mu)\int_\Omega  (u+1)^{{p}}}\\
\leq&\disp{\frac{({{p}-1})}{4}\int_{\Omega}(u+1)^{m+{{p}-3}}|\nabla u|^2+\chi^2({p}-1)C_1^2\int_\Omega (u+1)^{{p}}+
(C_2+\mu)\int_\Omega  (u+1)^{{p}}}\\
\leq&\disp{\frac{({{p}-1})}{4}\int_{\Omega}(u+1)^{m+{{p}-3}}|\nabla u|^2+C_3p\int_\Omega (u+1)^{{p}}-
\int_\Omega   (u+1)^{{p}}~~ \mbox{for all}~~~  t\in(0,T )}\\
\end{array}
\label{cz2aasweeettyyiii.5114114}
\end{equation}
with $C_1=\sup_{t\in(0,T)}\| v(\cdot,t)\|_{W^{1,\infty}(\Omega)},C_2=\xi\|w_0\|_{L^\infty(\Omega)}C_1+\kappa,C_3=\chi^2C_1^2+C_2+\mu+1$, where in the last inequality we have used the fact that
$\int_\Omega   \mu u(u+1)^{{p}-1}\leq \int_\Omega   \mu (u+1)^{{p}}$ and $u\geq0$.
Due to \dref{cz2aasweeettyyiii.5114114}, we deduce that
\begin{equation}
\begin{array}{rl}
&\disp{\frac{d}{dt}\|u+1\|^{{p}}_{L^{{p}}(\Omega)}+\int_\Omega   (u+1)^{{p}}+C_3\int_{\Omega}|\nabla (u+1)^{\frac{m+p-1}{2}}|^2\leq C_2p^2\int_\Omega (u+1)^{{p}}~~ \mbox{for all}~~~  t\in(0,T ).}\\
\end{array}
\label{cz2aasweeettyyiii.51sderftgg14114}
\end{equation}
Now,  we let ${l_0}>\max\{1,m-1\},p:=p_k = 2^k({l_0} + 1-m) + m-1$
and
\begin{equation}M_k =\max\{1,\sup_{t\in(0,T)}\int_{\Omega}(u+1)^{p_k}\}~~~\mbox{for}~~k\in \mathbb{N}.
\label{cz2aasweeettyyddrffgghiii.51sderftgg14114}
\end{equation}
We now invoke the Gagliardo--Nirenberg inequality ensures that 
\begin{equation}
\begin{array}{rl}
&C_2p_k^2\disp\int_\Omega (u+1)^{ p_k }\\
=&\disp{C_2p_k^2\|(u+1)^{\frac{m+p_k-1}{2}}\|_{L^{\frac{2 p_k }{m+ p_k -1}}(\Omega)}^{\frac{2 p_k }{m+ p_k -1}}}\\
\leq&\disp{C_4 p_k ^2(\|\nabla (u+1)^{\frac{m+ p_k -1}{2}}\|_{L^{2}(\Omega)}^{\frac{2 p_k }{m+ p_k -1}\varsigma_1}
\| (u+1)^{\frac{m+ p_k -1}{2}}\|_{L^{1}(\Omega)}^{\frac{2 p_k }{m+ p_k -1}(1-\varsigma_1)}+\| (u+1)^{\frac{m+ p_k -1}{2}}\|_{L^{1}(\Omega)}^{\frac{2 p_k }{m+ p_k -1}}),}\\
\end{array}
\label{cz2aasweeettyyiii.51sderftgg14114}
\end{equation}
where
$$\frac{2 p_k }{m+ p_k -1}\varsigma_1=\frac{2 p_k }{m+ p_k -1}\frac{N-\frac{N(m+ p_k -1)}{2 p_k }}{1-\frac{N}{2}+N}=\frac{2N( p_k +1-m)}{(N+2)(m+ p_k -1)}<2
$$
and
$$\frac{2 p_k }{m+ p_k -1}(1-\varsigma_1)=\frac{2 p_k }{m+ p_k -1}(1-\frac{N-\frac{N(m+ p_k -1)}{2 p_k }}{1-\frac{N}{2}+N})=2\frac{2 p_k +N(m-1)}{(N+2)(m+ p_k -1)}.
$$
Therefore, an application of the Young inequality yields
\begin{equation}
\begin{array}{rl}
C_2 p_k ^2\disp\int_\Omega (u+1)^{ p_k }\leq&\disp{C_5\|\nabla (u+1)^{\frac{m+ p_k -1}{2}}\|_{L^{2}(\Omega)}^{2}+C_6 p_k ^{\frac{(N+2)(m+ p_k -1)}{ p_k +(N+1)(m-1)}}
\| (u+1)^{\frac{m+ p_k -1}{2}}\|_{L^{1}(\Omega)}^{\frac{2 p_k +N(m-1)}{N(m-1)+m+ p_k -1}}}\\
&\disp{+C_7 p_k ^2\| (u+1)^{\frac{m+ p_k -1}{2}}\|_{L^{1}(\Omega)}^{\frac{2 p_k }{m+ p_k -1}}}\\
\leq&\disp{C_3\|\nabla (u+1)^{\frac{m+ p_k -1}{2}}\|_{L^{2}(\Omega)}^{2}+C_8 p_k ^{\frac{(N+2)(m+ p_k -1)}{ p_k +(N+1)(m-1)}}
\| (u+1)^{\frac{m+ p_k -1}{2}}\|_{L^{1}(\Omega)}^{\frac{2 p_k }{m+ p_k -1}}.}\\
\end{array}
\label{cz2aasweeettyyiiissdff.51sderftsdfffgg14114}
\end{equation}
Here we have used the fact that ${\frac{2 p_k +N(m-1)}{N(m-1)+m+ p_k -1}}\leq\frac{2 p_k }{m+ p_k -1}$ and ${\frac{(N+2)(m+ p_k -1)}{ p_k +(N+1)(m-1)}}\geq2$ (by $p_k>m-1$).
Thus, in light of $m\geq1,$  by means of \dref{cz2aasweeettyyddrffgghiii.51sderftgg14114}--\dref{cz2aasweeettyyiiissdff.51sderftsdfffgg14114}, 
\begin{equation}
\begin{array}{rl}
\disp\frac{d}{dt}\|u+1\|^{p_k}_{L^{p_k}(\Omega)}+\int_\Omega  (u+1)^{{p_k}}{}\leq&\disp{C_9 p_k ^{\frac{(N+2)(m+ p_k -1)}{ p_k +(N+1)(m-1)}}
\|   (u+1) ^{\frac{m+{p_k}-1}{2}}\|_{L^1(\Omega)}^{\frac{2 p_k }{m+ p_k -1}}}\\
\leq&\disp{\rho^k
M_{k-1}^{\frac{2 p_k }{m+ p_k -1}}}\\
\leq&\disp{\rho^k
M_{k-1}^{2}~~\mbox{for all}~~ t\in(0, T)}\\
\end{array}
\label{zjscz2.5297x9630111rrd67ddfff512df515}
\end{equation}
with some $\rho> 1.$
Here we have used the fact that
$${\frac{(N+2)(m+ p_k -1)}{ p_k +(N+1)(m-1)}}=\frac{2^k({l_0}+1-m)(N+2)+2(N+2)(m-1)}{2^k({l_0}+1-m)+(N+2)(m-1)}\leq N+2$$
and
$${\frac{2 p_k }{m+ p_k -1}}\leq{\frac{2 (p_k+m-1) }{m+ p_k -1}}=2.$$
Integrating \dref{zjscz2.5297x9630111rrd67ddfff512df515} over $(0, t)$ with $t\in(0, T)$, we derive
\begin{equation}
\begin{array}{rl}
\disp{\int_{\Omega}(u+1)^{p_k}(x,t)\leq\max\{\int_\Omega  (u_0+1) ^{{p_k}}{},\rho^k
M_{k-1}^{2}\}~~\mbox{for all}~~ t\in(0, T).}\\
\end{array}
\label{zjscz2.5297x9630111rrd67ddfff512ddfggghhhdf515}
\end{equation}
If $\int_{\Omega}(u+1)^{p_k}(x,t) \leq \int_\Omega  (u_0+1) ^{{p_k}}{}$ for any large $k\in \mathbb{N},$
then we obtain \dref{ssddaqwddfffhhhhkkswddaassffssff3.ddfvbb10deerfgghhjuuloollgghhhyhh} directly.
Otherwise,
by a straightforward induction, we have
\begin{equation}
\begin{array}{rl}
\disp\int_{\Omega} (u+1)^{p_k} {}\leq&\disp{
\rho^k
(\rho^{k-1}M_{k-2}^{2})^{2}}\\
=&\disp{\rho^{k+2(k-1)}M_{k-2}^{2^2}}\\
\leq&\disp{\rho^{k+\Sigma_{j=2}^k(j-1)}M_{0}^{2^k}.}\\
\end{array}
\label{cz2.56303hhyy890678789ty4tt8890013378}
\end{equation}
Combined with the boundedness of $M_0$ and in light of 
$\ln(1 + z)\leq z$ for
all $z\geq 0$,  so that,
taking $p_k$-th
roots on both sides of \dref{cz2.56303hhyy890678789ty4tt8890013378},  we can easily get \dref{ssddaqwddfffhhhhkkswddaassffssff3.ddfvbb10deerfgghhjuuloollgghhhyhh}.

Case $m<1$: Due to Lemmas \ref{lemmasdffssdd45566645630223}, \dref{lemdddmasdffssdd45566645630223} and \ref{lemmaddffddfffrsedrffffffgg}, we may choose
\begin{equation}
\tilde{p}_0:>\max\{6N(1-m),5(1-m),(3N+3)(1-m),5N\}
\label{ssdd999zjscz2.52ssdffssderrfgfgff97x96302222114}
\end{equation}
 such that
%
%
%
\begin{equation}
\int_{\Omega}(u+1)^{\tilde{p}_0}(x,t) \leq C_{10} ~~~\mbox{for all}~~ t\in(0,T).
\label{999zjscz2.52ssdffssderrfgfgff97x96302222114}
\end{equation}
Next,
testing the first equation in \dref{1.1} by $(u+1)^{p-1}$, integrating over $\Omega$, integrating by parts and applying the Young inequality and \dref{cz2sedfgg.5g5gghh56789hhjui7ssddd8jj90099}, we derive that
\begin{equation}
\begin{array}{rl}
&\disp{\frac{1}{{p}}\frac{d}{dt}\|u+1\|^{{p}}_{L^{{p}}(\Omega)}+C_{D}({{p}-1})\int_{\Omega}(u+1)^{m+{{p}-3}}|\nabla u|^2}
\\
\leq&\disp{-\chi\int_\Omega \nabla\cdot( u\nabla v)
  (u+1)^{{p}-1} -\xi\int_\Omega \nabla\cdot( u\nabla w)
  (u+1)^{{p}-1} +
\int_\Omega   (u+1)^{{p}-1}(\mu u-\mu u^2) }\\
\leq&\disp{\chi({p}-1)\int_\Omega  u(u+1)^{{p}-2}|\nabla u||\nabla v|
   -\xi({p-1})\int_\Omega \int_{0}^u \tau(\tau+1)^{p-2}d\tau\Delta w
   +
\int_\Omega   (u+1)^{{p}-1}(\mu u-\mu u^2) }\\
\leq&\disp{\chi({p}-1)\int_\Omega  u(u+1)^{{p}-2}|\nabla u||\nabla v|}\\
&\disp{+  \frac{\xi({p-1})}{p}\int_\Omega (u+1)^{p-1}(\tau \|w_0\|_{L^\infty(\Omega)}\cdot v(x,t)+\kappa)
   +
\int_\Omega   (u+1)^{{p}-1}(\mu u-\mu u^2) }\\
\leq&\disp{\frac{({{p}-1})}{4}\int_{\Omega}(u+1)^{m+{{p}-3}}|\nabla u|^2+\chi^2({p}-1)C_1^2\int_\Omega (u+1)^{{p}+1-m}+(C_2+\mu)\int_\Omega  (u+1)^{{p}}}\\
\leq&\disp{\frac{({{p}-1})}{4}\int_{\Omega}(u+1)^{m+{{p}-3}}|\nabla u|^2+\chi^2({p}-1)C_1^2\int_\Omega (u+1)^{{p}+1-m}+
(C_2+\mu)\int_\Omega  (u+1)^{{p}+1-m}}\\
\leq&\disp{\frac{({{p}-1})}{4}\int_{\Omega}(u+1)^{m+{{p}-3}}|\nabla u|^2+C_{10}p\int_\Omega (u+1)^{{p}+1-m}-
\int_\Omega   (u+1)^{{p}}~~ \mbox{for all}~~~  t\in(0,T ),}\\
\end{array}
\label{999cz2aasweeettyyiii.5114ddffgg114}
\end{equation}
with $C_{1}=\sup_{t\in(0,T)}\| v(\cdot,t)\|_{W^{1,\infty}(\Omega)},C_{2}=\xi\|w_0\|_{L^\infty(\Omega)}C_{1}+\kappa,C_{11}=\chi^2C_1^2+C_{2}+\mu+1$.
Here we have used the fact that $\int_\Omega    u(u+1)^{{p}-1}\leq \int_\Omega    (u+1)^{{p}}\leq\int_\Omega (u+1)^{{p}+1-m}$ and $u\geq0$.
Therefore, \dref{999cz2aasweeettyyiii.5114ddffgg114} yields to
\begin{equation}
\begin{array}{rl}
&\disp{\frac{d}{dt}\|u+1\|^{{p}}_{L^{{p}}(\Omega)}+\int_\Omega   (u+1)^{{p}}+
C_{12}\int_{\Omega}|\nabla (u+1)^{\frac{m+p-1}{2}}|^2\leq C_{13}p^2\int_\Omega (u+1)^{{p}+1-m}}\\
\end{array}
\label{3drftgyyyyyuuucz2aasweeettyyiii.51sderftgg14114}
\end{equation}
for all $t\in(0,T).$
Let  $p:=\tilde{p}_k = 2^k(\tilde{p}_0 + 1-m) + m-1$
and
\begin{equation}\tilde{M}_k =\max\{1,\sup_{t\in(0,T)}\int_{\Omega}(u+1)^{\tilde{p}_k}\}~~~\mbox{for}~~k\in \mathbb{N},
\label{999cz2aasweeettyyddrffgghiii.51sderftgg14114}
\end{equation}
where $\tilde{p}_0$ is given by \dref{ssdd999zjscz2.52ssdffssderrfgfgff97x96302222114}.
As moreover by the Gagliardo¨CNirenberg inequality, we have
\begin{equation}
\begin{array}{rl}
&C_{13}\tilde{p}_k^2\disp\int_\Omega (u+1)^{ \tilde{p}_k +1-m}\\
=&\disp{C_{13}\tilde{p}_k^2\|(u+1)^{\frac{m+\tilde{p}_k-1}{2}}\|_{L^{\frac{2(\tilde{p}_k+1-m) }{m+ \tilde{p}_k -1}}(\Omega)}^{\frac{2(\tilde{p}_k+1-m) }{m+ \tilde{p}_k -1}}}\\
\leq&\disp{C_{14} \tilde{p}_k ^2(\|\nabla (u+1)^{\frac{m+ \tilde{p}_k -1}{2}}\|_{L^{2}(\Omega)}^{\frac{2(\tilde{p}_k+1-m) }{m+ \tilde{p}_k -1}\varsigma_2}
\| (u+1)^{\frac{m+ \tilde{p}_k -1}{2}}\|_{L^{1}(\Omega)}^{\frac{2(\tilde{p}_k+1-m) }{m+ \tilde{p}_k -1}(1-\varsigma_2)}+\| (u+1)^{\frac{m+ \tilde{p}_k -1}{2}}\|_{L^{1}(\Omega)}^{\frac{2(\tilde{p}_k+1-m) }{m+ \tilde{p}_k -1}}),}\\
\end{array}
\label{999cz2aasweeettyyiii.51sderftgg14114}
\end{equation}
where
$$\begin{array}{rl}
\disp\frac{2(\tilde{p}_k+1-m) }{m+ \tilde{p}_k -1}\varsigma_2=&\disp\frac{2(\tilde{p}_k+1-m) }{m+ \tilde{p}_k -1}\frac{N-\frac{N(m+ \tilde{p}_k -1)}{2(\tilde{p}_k+1-m) }}{1-\frac{N}{2}+N}\\
=&\disp\frac{2N\tilde{p}_k +6N(1-m)}{(m+ \tilde{p}_k -1)(N+2)}<2~\mbox{by}~\tilde{p}_k>(2N+1)(1-m)\\
\end{array}
$$
and
$$\frac{2(\tilde{p}_k+1-m) }{m+ \tilde{p}_k -1}(1-\varsigma_2)=\frac{2(\tilde{p}_k+1-m) }{m+ \tilde{p}_k -1}(1-\frac{N-\frac{N(m+ \tilde{p}_k -1)}{2(\tilde{p}_k+1-m) }}{1-\frac{N}{2}+N})=4\frac{\tilde{p}_k+(m-1)(N-1)}{(m+\tilde{p}_k-1)(N+2)}.
$$
Therefore, in light  of the Young inequality, we conclude that
\begin{equation}
\begin{array}{rl}
&C_{13} \tilde{p}_k ^2\disp\int_\Omega (u+1)^{ \tilde{p}_k+1-m }\\
\leq&\disp{C_{12}\|\nabla (u+1)^{\frac{m+ \tilde{p}_k -1}{2}}\|_{L^{2}(\Omega)}^{2}+C_{15} \tilde{p}_k ^{\frac{(m+ \tilde{p}_k -1)(N+2)}{ \tilde{p}_k +(m-1)(2N+1)}}
\| (u+1)^{\frac{m+ \tilde{p}_k -1}{2}}\|_{L^{1}(\Omega)}^{\frac{2[\tilde{p}_k +(N-1)(m-1)]}{\tilde{p}_k +(2N+1)(m-1)}}}\\
&\disp{+C_{16} \tilde{p}_k ^2\| (u+1)^{\frac{m+ \tilde{p}_k -1}{2}}\|_{L^{1}(\Omega)}^{\frac{2 (\tilde{p}_k +1-m)}{m+ \tilde{p}_k -1}}}\\
\leq&\disp{C_{12}\|\nabla (u+1)^{\frac{m+ \tilde{p}_k -1}{2}}\|_{L^{2}(\Omega)}^{2}+C_{17} \tilde{p}_k ^{\frac{(m+ \tilde{p}_k -1)(N+2)}{ \tilde{p}_k +(m-1)(2N+1)}}
\| (u+1)^{\frac{m+ \tilde{p}_k -1}{2}}\|_{L^{1}(\Omega)}^{\frac{2 (\tilde{p}_k+1-m) }{m+ \tilde{p}_k -1}},}\\
\end{array}
\label{999cz2aasweeettyyiiissdff.51sderftsdfffgg14114}
\end{equation}
where we have utilized the following facts
 $$\frac{2[\tilde{p}_k +(N-1)(m-1)]}{\tilde{p}_k +(2N+1)(m-1)}\geq\frac{2 (\tilde{p}_k+1-m) }{m+ \tilde{p}_k -1}~~~ \mbox{and}~~
 \frac{(m+ \tilde{p}_k -1)(N+2)}{ \tilde{p}_k +(m-1)(2N+1)}\geq2.$$
 The fact $\tilde{p}_0>(1-m)(4N+1)$ then ensures
 $$\begin{array}{rl}
{\disp\frac{(m+ \tilde{p}_k -1)(N+2)}{ \tilde{p}_k +(m-1)(2N+1)}}=&\disp{(N+2)\frac{2^k({\tilde{p}_0}+1-m)+2(m-1)}{2^k({\tilde{p}_0}+1-m)+(2N+2)(m-1)}}\\
\leq&\disp{
(N+2)\frac{{\tilde{p}_0}+1-m+2(m-1)}{{\tilde{p}_0}+1-m+(2N+2)(m-1)}}\\
\leq&\disp{2(N+2),}\\
\end{array}
$$
so that, in light of
 \dref{ssdd999zjscz2.52ssdffssderrfgfgff97x96302222114},
\dref{999cz2aasweeettyyddrffgghiii.51sderftgg14114}--\dref{999cz2aasweeettyyiiissdff.51sderftsdfffgg14114}, 
\begin{equation}
\begin{array}{rl}
\disp&\disp\frac{d}{dt}\|u+1\|^{\tilde{p}_k}_{L^{\tilde{p}_k}(\Omega)}+\int_\Omega  (u+1) ^{{\tilde{p}_k}}{}\\
\leq&\disp{C_{18} \tilde{p}_k ^{\frac{(m+ \tilde{p}_k -1)(N+2)}{ \tilde{p}_k +(m-1)(2N+1)}}
\|   (u+1) ^{\frac{m+{\tilde{p}_k}-1}{2}}\|_{L^1(\Omega)}^{\frac{2[\tilde{p}_k +(N-1)(m-1)]}{\tilde{p}_k +(2N+1)(m-1)}}}\\
\leq&\disp{\tilde{\rho}^k
\tilde{M}_{k-1}^{\frac{2[\tilde{p}_k +(N-1)(m-1)]}{\tilde{p}_k +(2N+1)(m-1)}}~~\mbox{for all}~~ t\in(0, T)}\\
\end{array}
\label{999zjscz2.5297x9630111rrd67ddfff512df515}
\end{equation}
with some $\tilde{\rho}> 1,$ where $$\begin{array}{rl}
&{\disp\frac{2[\tilde{p}_k +(N-1)(m-1)]}{\tilde{p}_k +(2N+1)(m-1)}}\\
=&2{\frac{2^k (\tilde{p}_0+1-m)+N(m-1)}{2^k (\tilde{p}_0+1-m)+(2N+2)(m -1)}}\\
=&2(1+{\frac{(N+2)(1-m) }{2^k (\tilde{p}_0+1-m)+(2N+2)(m -1)}}):=\kappa_k.\\
\end{array}$$
%
Here we note that $\kappa_k=2(1+\varepsilon_k)$ for $k\geq1$,  where $\varepsilon_k$ satisfies $\varepsilon_k\leq \frac{C_{19}}{2^{k}}$ for all $k$ with some $C_{17}>0$.
Next, we
integrate \dref{999zjscz2.5297x9630111rrd67ddfff512df515} over $(0, t)$ with $t\in(0, T)$, then yields to
\begin{equation}
\begin{array}{rl}
\disp{\int_{\Omega}(u+1)^{\tilde{p}_k}(x,t)\leq\max\{\int_\Omega  (u_0+1) ^{{\tilde{p}_k}}{},\tilde{\rho}^k
\tilde{M}_{k-1}^{\frac{2[\tilde{p}_k +(N-1)(m-1)]}{\tilde{p}_k +(2N+1)(m-1)}}\}~~\mbox{for all}~~ t\in(0, T).}\\
\end{array}
\label{999zjscz2.5297x9630111rrd67ddfff512ddfggghhhdf515}
\end{equation}
If $\int_{\Omega}(u+1)^{\tilde{p}_k}(x,t) \leq \int_\Omega  (u_0+1) ^{{\tilde{p}_k}}{}$ for any large $k\in \mathbb{N},$
then we derive  \dref{ssddaqwddfffhhhhkkswddaassffssff3.ddfvbb10deerfgghhjuuloollgghhhyhh} holds.
Otherwise,
by a straightforward induction, we have
\begin{equation}
\begin{array}{rl}
\disp\int_{\Omega} (u+1)^{\tilde{p}_k} {}\leq&\disp{
\tilde{\rho}^{k+\sum_{j=2}^k(j-1)\cdot\prod_{i=j}^k\kappa_i}
\tilde{M}_0^{\prod_{i=1}^k\kappa_i}~~\mbox{for all}~~k\geq1.}\\
\end{array}
\label{ssdd4444cz2.56303hhyy890678789ty4tt8890013378}
\end{equation}
On the other hand, due to the fact that $\ln(1+x)\leq x$ (for all $x\geq0$), a simple computation yields
$$\begin{array}{rl}
\disp\prod_{i=j}^k\kappa_i
=&2^{k+1-j}e^{\Sigma_{i=j}^k\ln(1+\varepsilon_j)}\\
\leq&2^{k+1-j}e^{\Sigma_{i=j}^k \varepsilon_j}\\
\leq&2^{k+1-j}e^{C_{20}}~~~\mbox{for all}~~k\geq1~~\mbox{and}~~j=\{1,\ldots,k\}.\\
\end{array}$$
 This together with \dref{ssdd4444cz2.56303hhyy890678789ty4tt8890013378} entails that
\begin{equation}
\begin{array}{rl}
\disp\left(\int_{\Omega} (u+1)^{\tilde{p}_k}\right)^{\frac{1}{\tilde{p}_k}} {}\leq&\disp{
\tilde{\rho}^{\frac{k}{\tilde{p}_k}+\frac{\sum_{j=2}^k(j-1)\cdot\prod_{i=j}^k\kappa_i}{\tilde{p}_k}}
\tilde{M}_0^{\frac{\prod_{i=1}^k\kappa_i}{\tilde{p}_k}}~~\mbox{for all}~~k\geq1,}\\
\end{array}
\label{ssdd4444ddfffcz2.56303hhyy890678789ty4tt8890013378}
\end{equation}
which after taking $k\rightarrow\infty$ readily implies that \dref{ssddaqwddfffhhhhkkswddaassffssff3.ddfvbb10deerfgghhjuuloollgghhhyhh}  is valid.
\end{proof}

The previous lemmas at hand, we can conclude main results in a straightforward manner.
{\bf The proof of main results}~
Theorem \ref{dfffftheorem3} (and Theorem  \ref{theorem3}) will be proved if we can show $T_{max}=\infty$. Suppose on contrary that $T_{max}<\infty$.
%
In view of \dref{ssddaqwddfffhhhhkkswddaassffssff3.ddfvbb10deerfgghhjuuloollgghhhyhh}, we apply Lemma \ref{lemma70} to reach a contradiction.
 Hence the  classical solution $(u,v,w)$ of \dref{1.1} is global in time and bounded. $\qed$

{\bf Acknowledgement}:
{\bf Acknowledgement}:
This work is partially supported by Shandong Provincial
Science Foundation for Outstanding Youth (No. ZR2018JL005).


\end{document}